%% file: main.tex
\documentclass[11pt,a4paper]{article}

\usepackage[utf8]{inputenc}
\usepackage{amsmath, amsthm, amssymb, mathtools}
\usepackage{geometry}
\usepackage{graphicx}
\usepackage{float}
\usepackage[dvipsnames]{xcolor}
\definecolor{navy}{RGB}{27, 42, 74}
\usepackage{tikz}
\usetikzlibrary{matrix}
\usepackage{tikz-3dplot}

\usepackage{booktabs}
\usepackage{multirow}

\usepackage[numbers,sort&compress]{natbib}
\usepackage{doi}
\usepackage{hyperref}
\hypersetup{colorlinks=true, linkcolor=blue, citecolor=blue, urlcolor=blue}

\newtheorem{theorem}{Theorem}[section]
\newtheorem{lemma}[theorem]{Lemma}

\newtheorem{corollary}[theorem]{Corollary}
\newtheorem{remark}[theorem]{Remark}
\newtheorem{definition}[theorem]{Definition}

\DeclareMathOperator{\GL}{GL}
\DeclareMathOperator{\Hom}{Hom}
\DeclareMathOperator{\tr}{tr}

\DeclareMathOperator{\sgn}{sgn}

\title{\LARGE \textbf{Parallelisation of Discrete Exterior Calculus via Representation Theory on Curved and Three-Dimensional Meshes}}

\author{\large \textbf{Leon D. da Silva}, \textbf{Marcelo P. Santos}, \textbf{Jos\'{e} D. da Silva}, \textbf{Gilson Ferreira Jr.} \\
\small Departamento de Matem\'{a}tica, Universidade Federal Rural de Pernambuco \\
\small Rua Dom Manoel de Medeiros, s/n, Recife, PE, 52171-900, Brazil \\
\small Corresponding author: \texttt{leon.silva@ufrpe.br}}

\date{\small \today}

\begin{document}

\maketitle

\begin{center}
\resizebox{\linewidth}{!}{%
    \begin{tikzpicture}[font=\sffamily, line cap=round, line join=round, >=stealth, thick]

      \node[font=\sffamily\bfseries, navy] at (0.2,3.25) {1.\ Symmetric Mesh};
      \node[font=\sffamily\small, gray!60] at (0.2,2.65) {Topology \& group $G$};

      \begin{scope}[shift={(-0.55,0.15)}, scale=0.72]
        \shade[ball color=gray!14, opacity=0.45] (0,0) circle[radius=1.92];
        \tdplotsetmaincoords{70}{110}
        \begin{scope}[tdplot_main_coords]
          \input{figures/icosahedron_mesh.tex}
        \end{scope}
      \end{scope}
      \node[font=\sffamily\scriptsize, gray!70] at (-0.55,-1.85) {$S^2$\,:\ $I_h$};

      \tdplotsetmaincoords{70}{120}
      \begin{scope}[shift={(1.80,-0.1)}, scale=0.64, tdplot_main_coords]
        \coordinate (P0) at (-1,-1,-1);
        \coordinate (P1) at ( 1,-1,-1);
        \coordinate (P2) at ( 1, 1,-1);
        \coordinate (P3) at (-1, 1,-1);
        \coordinate (P4) at (-1,-1, 1);
        \coordinate (P5) at ( 1,-1, 1);
        \coordinate (P6) at ( 1, 1, 1);
        \coordinate (P7) at (-1, 1, 1);
        \coordinate (Oc) at ( 0, 0, 0);
        \fill[navy!7] (P4)--(P5)--(P6)--(P7)--cycle;
        \draw[gray!35, line width=0.5pt] (P0)--(P1) (P0)--(P3) (P0)--(P4);
        \draw[navy!20, line width=0.35pt]
            (Oc)--(P0) (Oc)--(P1) (Oc)--(P2) (Oc)--(P3)
            (Oc)--(P4) (Oc)--(P5) (Oc)--(P6) (Oc)--(P7);
        \draw[gray!75, line width=0.7pt]
            (P1)--(P2)--(P3) (P4)--(P5)--(P6)--(P7)--cycle
            (P1)--(P5) (P2)--(P6) (P3)--(P7);
        \draw[gray!55, line width=0.5pt, densely dashed] (P4)--(P6);
        \fill[navy] (Oc) circle[radius=1.7pt];
      \end{scope}
      \node[font=\sffamily\scriptsize, gray!70] at (1.80,-1.85) {$T^3$\,:\ $T_d$};

      \draw[->, line width=1.3pt, gray!55] (3.40,0) -- (4.35,0);

      \node[font=\sffamily\bfseries, navy] at (5.95,3.25) {2.\ Universal Equivariance};
      \node[font=\normalsize] at (5.95,2.45)
            {$[\rho(g),d]=0 \quad\&\quad [\rho(g),\star]=0$};
      \node[font=\sffamily\small, gray!60] at (5.95,1.75)
            {single symmetry-adapted basis $Q$};
      \node[font=\sffamily\small, gray!60] at (5.95,1.30)
            {(built once, PDE-agnostic)};
      \node[font=\normalsize] at (5.95,0.0)
            {$L \mapsto Q^{\mathsf{T}} L Q$};

      \draw[->, line width=1.3pt, gray!55] (7.55,0) -- (8.50,0);

      \node[font=\sffamily\bfseries, navy] at (13.05,3.25) {3.\ Decoupled Systems};
      \node[font=\sffamily\small, gray!60] at (13.05,2.65) {block-diagonal, solved in parallel};

      \node[anchor=west, font=\sffamily\small] at (9.55,1.0)  {Poisson $(\Delta)$};
      \node[anchor=west, font=\sffamily\small] at (9.55,0.3)  {Maxwell $(d,\delta)$};
      \node[anchor=west, font=\sffamily\small] at (9.55,-0.4) {Navier--Stokes$^{\ast}$};
      \draw[gray!60, line width=0.8pt]
            (9.45,1.25) -- (9.32,1.25) -- (9.32,-0.65) -- (9.45,-0.65);
      \draw[->, line width=1.0pt, gray!60] (12.05,0.30) -- (12.85,0.30);

      \begin{scope}[shift={(13.05,1.55)}, scale=0.80]
        \fill[gray!4] (0,0) rectangle (3.6,-3.6);
        \fill[navy!58] (0.0, 0.0) rectangle (1.1,-1.1);
        \node[white, font=\bfseries\small] at (0.55,-0.55) {$L_1$};
        \fill[navy!72] (1.1,-1.1) rectangle (2.5,-2.5);
        \node[white, font=\bfseries\small] at (1.8,-1.8) {$L_2$};
        \node[gray!75, font=\large] at (2.83,-2.83) {$\ddots$};
        \fill[navy!85] (3.0,-3.0) rectangle (3.6,-3.6);
        \node[white, font=\bfseries\scriptsize] at (3.3,-3.3) {$L_n$};
        \draw[gray!70, line width=0.9pt] (0,0) rectangle (3.6,-3.6);
      \end{scope}

      \node[anchor=west, font=\sffamily\scriptsize, gray!70] at (9.55,-1.55)
            {$^{\ast}$linear substeps (operator splitting)};

    \end{tikzpicture}
    }
\end{center}

\noindent\textbf{Graphical Abstract.} A single symmetry-adapted orthogonal basis $Q$, built once from the mesh symmetry group $G$, block-diagonalizes every DEC operator $L$ into independently solvable isotypic blocks via $L \mapsto Q^{\mathsf{T}} L\, Q$.

\medskip

\begin{abstract}
\noindent We establish a universal block-diagonalization framework for Discrete Exterior Calculus (DEC) operators on symmetric meshes, enabling embarrassingly parallel solvers with provable FLOP reductions. We prove that the two fundamental DEC operators, the discrete exterior derivative $d$ and the Hodge star $\star$, are equivariant under isometric finite group actions on simplicial complexes. The proof exploits the permutation representation induced on cochain spaces by the group action. As a consequence, any operator assembled from $d$ and $\star$ (including the Hodge Laplacian, the codifferential, Maxwell-type operators, and elasticity operators) inherits a block-diagonal structure in a single symmetry-adapted basis, which is computed only once per mesh. Unlike spectral methods restricted to flat Platonic domains, the framework applies natively to curved manifolds and is applicable in principle to computational electromagnetism and geometric fluid simulation on symmetric domains. Numerical experiments on a geodesic sphere ($I_h$ symmetry) and a hexagonal torus ($D_{6h}$ symmetry) yield FLOP-based parallel speedups, relative to a dense direct factorization, of up to $62\times$ and $182\times$, respectively. A further experiment on a body-centred-cubic (BCC) tessellation of the flat 3-torus $T^3$ with $T_d$ symmetry confirms equivariance of the exterior derivative, Hodge star, and Hodge Laplacian at machine precision for form degrees $k=0,1,2$ across three mesh resolutions. The FLOP-based sequential speedup approaches its theoretical asymptote of $\approx 9.07\times$, which a standard Schur-multiplicity reduction deepens by a further factor of order $|G|$. These results show that a single symmetry-adapted basis reduces the linear-solve cost of structure-preserving DEC computations on curved and three-dimensional meshes.

\vspace{1em}

\vspace{0.5em}
\end{abstract}

\section{Introduction}

Discrete Exterior Calculus (DEC) \cite{hirani2003,desbrun2005} is a structure-preserving framework for the discretization of partial differential equations on manifolds. DEC represents physical fields as discrete differential forms (cochains) defined on simplicial complexes and their circumcentric duals. This construction guarantees that the generalized Stokes' theorem, the de~Rham complex structure, and physical conservation laws hold exactly at the discrete level, avoiding non-physical artifacts common to conventional discretizations. The framework is applied in computational electromagnetism \cite{bossavit1998} and geometry processing \cite{crane2013}. DEC also extends naturally to triangulated curved surfaces. Fluid solvers in this framework originate in a circulation-preserving simplicial scheme~\cite{elcott2007}. The incompressible Navier--Stokes equations have since been discretized on planar and spherical surfaces~\cite{mohamed2016,jagad2021}, and a covariant formulation treats genus-zero flows together with harmonic vector fields on the torus~\cite{nitschke2017}. Numerical convergence is established on arbitrary, including non-Delaunay, surface meshes~\cite{mohamed2018}.

Despite these structural guarantees, achieving high-fidelity numerical resolution in DEC requires inverting large, sparse matrix systems---most notably those arising from the discrete Hodge Laplacian $\Delta = \delta d + d\delta$. Even scalable iterative solvers such as multigrid can require problem-specific tuning here, owing to the kernels induced by nontrivial cohomology and to the anisotropy of meshes over curved manifolds.

A well-established paradigm for reducing these linear solve costs is \emph{symmetry exploitation}. When a spatial discretization respects the symmetries of a finite group $G$, classical representation theory guarantees a block-diagonalization of the system matrix via a unitary change to a \emph{symmetry-adapted basis}, decoupling the problem into independent subsystems of reduced dimension. This algebraic reduction has historical roots in computational chemistry and molecular physics \cite{allgower1997,fassler1992}. Bossavit formalised the construction for boundary-value problems, reducing a $G$-invariant problem to one subproblem per irreducible representation posed on the symmetry cell~\cite{bossavit1986}. The same reduction admits a purely algebraic statement, in which the generalized Fourier transform block-diagonalizes equivariant matrices~\cite{ahlander2005}.

Representation theory is now used to exploit symmetries across a wide range of computational methodologies. Olver~\cite{olver2025} recently demonstrated how global polynomial spectral bases on flat Platonic domains (rectangles, cubes) produce block-diagonalizations that reduce the cost of linear solves. However, such spectral approaches are conceptually tied to flat geometries. As Olver himself emphasizes~\cite[Sec.~2]{olver2025}, meshes that faithfully carry a finite symmetry group in 3D are essentially restricted to variants of the five Platonic solids; combined with the scalability of multigrid solvers on general meshes, this geometric restriction led the computational community to largely abandon symmetry-adapted bases in mesh-based discretizations. The situation is more severe on curved manifolds, where global polynomial bases are unavailable altogether.

Similarly, within the framework of Finite Element Exterior Calculus (FEEC)~\cite{arnold2006,arnold2010}, representation theory classifies symmetry-invariant bases of polynomial differential forms on reference simplices~\cite{berchenko2024,licht2024}, yet the challenge of establishing the equivariance of globally assembled operators on the mesh remains. A related but distinct use of symmetry appears in discrete geometric mechanics, where discrete Noether-type theorems ensure that variational integrators preserve momentum structures associated with the underlying symmetries~\cite{marsden2001}. Despite these advances, no existing framework, whether in spectral methods, FEEC, geometric deep learning~\cite{bronstein2021,cohen2016}, or variational integrators, has so far established the equivariance of mesh-based discrete operators, such as those of DEC, under finite group actions on curved manifolds.

We close this gap by establishing a block-diagonalization framework for Discrete Exterior Calculus. Our contributions are twofold. First, we prove strict equivariance for the two fundamental DEC operators, the discrete exterior derivative $d$ and the discrete Hodge star $\star$, under isometric group actions on simplicial complexes of arbitrary topological dimension. The proof of equivariance for $d$ is purely combinatorial and dimension-agnostic, and the Hodge-star equivariance holds under the same isometry hypothesis. The dimensional subtlety concerns the invertibility of $\star$ rather than its equivariance. On any strictly Delaunay surface the circumcentric star is positive, whereas in three dimensions positivity is a strict additional requirement, met here by a centroid-based star (Section~\ref{sec:exp_3complex}). Because \emph{every} DEC-constructed differential operator is assembled from $d$ and $\star$, the block-diagonal structure is inherited automatically. A single symmetry-adapted basis block-diagonalizes all DEC operators on the mesh.

Second, this reach is geometric as well as dimensional. Whereas spectral polynomial bases require flat, regular (Platonic) domains, the DEC construction carries the symmetry group of any simplicial mesh admitting a finite isometry group, from curved surfaces such as geodesic spheres and hexagonal tori to three-dimensional tessellations such as the body-centred-cubic lattice on the 3-torus.

This algebraic reduction complements existing mesh-based parallelization rather than replacing it, composing naturally with MPI-based domain decomposition on unstructured tetrahedralizations~\cite{boom2022}.

The symmetry-adapted basis is computed by character-theoretic projection on sparse cochain spaces, requiring only the combinatorial indexing of the group action and no polynomial spatial algebra. Across curved two-dimensional surfaces and a three-dimensional tessellation, equivariance holds to machine precision and the sequential block-diagonal speedup approaches its regular-representation limit, with $\Delta_2$ saturating this limit already at modest mesh resolution, and a standard Schur-multiplicity reduction deepens it by a further factor of order $|G|$ (Section~\ref{sec:experiments}).

Sections~\ref{sec:finding_symmetry_adapted_basis}--\ref{sec:formalism_dec} review the algebraic and geometric foundations: representation theory, symmetry-adapted bases, and the DEC formalism with isometric group actions. Sections~\ref{sec:equivariance}--\ref{sec:universal_block} prove the equivariance of $d$ and $\star$ and derive the universal block-diagonalization theorem. Sections~\ref{sec:experiments}--\ref{sec:conclusion} present numerical experiments on curved two-dimensional surfaces and three-dimensional tessellations, discuss implications, and conclude.

\section{Representation-Theoretic Tools for Symmetry Exploitation}
\label{sec:finding_symmetry_adapted_basis}
This section provides a concise outline of the symmetry-adapted basis method within the representation theory of finite groups, aimed at the block-diagonalization of equivariant operators. For brevity, we omit full demonstrations; interested readers can find detailed proofs and extended discussions in standard references such as \cite{serre1977} (Chapters 1 to 5) and \cite{fassler1992}.

Let $V$ be a finite-dimensional vector space over $\mathbb{C}$ (or $\mathbb{R}$ when representations are strictly real-orthogonal, as is typical in geometric discretizations). Let $G$ be a finite group. A homomorphism $\phi: G \to \GL(V)$ is a \emph{linear representation} of $G$. When $\phi$ is given, we say that $V$ is a representation of $G$. The dimension of $V$ is called the \emph{degree} of $\phi$, denoted as $\deg(\phi)$.

A vector subspace $U \subset V$ is said to be $\phi$-invariant if $\phi(g)(U) \subset U$ for all $g \in G$. A $\phi$-invariant subspace $U \subset V$ is a \emph{subrepresentation} of $\phi$; thus, a subrepresentation is itself a representation. A representation $\phi$ is \emph{reducible} if there is a $\phi$-invariant proper subspace $U \subset V$, and \emph{irreducible} otherwise.

Under our assumptions, a representation $\phi$ has a decomposition $V = \bigoplus_{i=1}^{r} U_i$ such that the subrepresentation $\phi_i(g) = \phi(g)|_{U_i}$ for all $g \in G$ is irreducible. Accordingly, we can express $\phi$ as a direct sum of irreducible representations (irreps), $\phi = \phi_{1} \oplus \cdots \oplus \phi_{r}$. In matrix form, this decomposition induces a block-diagonal matrix structure, with each block corresponding to one irrep $\phi_i$.

Given two representations $\phi: G \to \GL(V)$ and $\psi: G \to \GL(W)$, a linear transformation $\tau: V \to W$ such that 
\begin{equation}
\psi(g) \circ \tau = \tau \circ \phi(g), \quad \forall g \in G,     
\end{equation}
is called an \emph{equivariant map}. If $\tau$ is an isomorphism, we say $\psi$ and $\phi$ are \emph{equivalent}.

Consider the decomposition of a representation $\phi: G \to \GL(V)$ into its irreducible subrepresentations. 
We can write
\begin{equation}
\label{eq:isotypic_decomposition}
  V = V_{1} \oplus \cdots \oplus V_{n},  
\end{equation}
where each $V_j$ is the sum 
\begin{equation}
V_j = \bigoplus_{l=1}^{c_j} V_{j}^{l} \label{eq:isotypic_component}
\end{equation}
of equivalent subrepresentations.
Here, $c_j$ is called the \emph{multiplicity}, while the common degree of each $V_j^l$ is denoted by $n_j$. The $(c_j n_j)$-dimensional subspaces $V_j$ are called \emph{isotypic components}, and \eqref{eq:isotypic_decomposition} is called the isotypic decomposition. 

The decompositions \eqref{eq:isotypic_decomposition} and \eqref{eq:isotypic_component} induce a decomposition 
\begin{equation}
\phi = c_1\phi_{1} \oplus \cdots \oplus c_n\phi_{n},
\label{eq:phi_decomposition}
\end{equation}
where $\phi_j$ is the restriction of $\phi$ to one irreducible component $V_j^l$, and $c_j$ denotes the number of copies.

A basis of $V$ that realizes the decomposition \eqref{eq:phi_decomposition} is called a \textit{symmetry-adapted basis}. In general, it is not unique. In the next subsection, we describe an algorithm to obtain one such basis.

\subsection{Calculating the Symmetry-Adapted Basis}

Let $\chi_{\phi}(g) = \tr(\phi(g))$. The function $\chi: G \to \mathbb{C}$ is the \emph{character} of $\phi$. The canonical inner product between characters is:
\begin{equation}
\label{eq:Produto_interno_dos_Carateres}
    \langle \chi_1, \chi_2 \rangle = \frac{1}{|G|}\sum_{g\in G} \chi_1(g^{-1}) \chi_2(g).
\end{equation}  
The set of irreducible subrepresentations of each finite group $G$ is finite (up to equivalence). Given a set of representative elements of each class of nonequivalent representations, we say that the corresponding characters form a complete set of characters. Let $\{\chi_1,\ldots,\chi_k\}$ be a complete set of characters of $G$.

Characters provide a method for identifying components because the inner product $\langle \chi_{\phi}, \chi_i \rangle$ is equal to the multiplicity $c_i$, denoting exactly how many distinct copies of an irreducible representation possessing character $\chi_i$ occur within $\phi$. 

We can define the projection operators $\Pi_j$, for $j=1,\ldots, k$, corresponding to the projection of $V$ onto $V_j$. The explicit expressions for these projectors are given below:
\begin{equation}
\Pi_j = \frac{n_j}{|G|} \sum_{g \in G} \chi_j(g^{-1})\phi(g).
\label{eq:isotypicDecomposition}
\end{equation}
In the above expression, the coefficient $\frac{n_j}{|G|}$ is not required, and we can choose to omit it to reduce the computational cost, or include it to make the associated operators orthonormal. Also, if necessary, we can modify the projectors to make the symmetry-adapted basis orthonormal. In our case, no such re-orthonormalization is required in exact arithmetic, since the representation is already given by orthonormal operators; the numerical construction of Section~\ref{subsec:algorithm} nonetheless restores orthonormality lost to floating-point round-off.

Given a complete set of irreducible representations $\delta_j$, for $j=1,\ldots, k$, of $G$, and letting $(d^j_{pq})$ be the corresponding matrix representation, one defines the \textit{transfer operators}:
\begin{equation}
P^{j}_{pq} = \frac{n_j}{|G|}\sum_{g \in G} d^{j}_{pq}(g^{-1})\phi(g). \label{eq:transferences}
\end{equation}
The linear operator $P^j_{pq}$ is null for an isotypic component $V_l$ ($l \neq j$) and acts as an isomorphism between the irreducible components of $V_j$. The symmetry-adapted basis is explicitly constructed by means of these operators.

To systematically construct the complete symmetry-adapted basis, we can follow these steps: 
\begin{itemize}
    \item[(i)] Construct $P^j_{11}$, whose image has dimension $c_j$ (the multiplicity of $\phi_j$). Select a basis $\beta_j^1 = \{v_1^1,v_1^2,\ldots,v_1^{c_j}\}$ of its image.
    \item[(ii)] For $m=2,\ldots,n_j$ and $i=1,\ldots,c_j$, define the vectors:
    \begin{equation}
  v_m^i = P^j_{1m}v_1^i.\label{eq:Equation_for_transferences}
    \end{equation}
    Applying $P^j_{12}$ to each vector in $\beta_j^1$, we obtain $\beta_j^2 = \{v_2^1,v_2^2,\ldots,v_2^{c_j}\}$. Applying $P^j_{13}$ to each vector in $\beta_j^1$, we obtain $\beta_j^3 = \{v_3^1,v_3^2,\ldots,v_3^{c_j}\}$. And so on.
    
The ordered basis 
\begin{align}
\beta_j &=   \beta_j^1 \cup \cdots \cup \beta_j^{n_j}\\
&=\{v_1^1,v_1^2,\ldots,v_1^{c_j}, v_2^1,v_2^2,\ldots,v_2^{c_j},\ldots, v_{n_j}^1,v_{n_j}^2,\ldots,v_{n_j}^{c_j}\}  
\end{align}
is the part of the symmetry-adapted basis relative to the isotypic component $V_j$. Now, we repeat this process from $j=1$ to $j=k$ to complete the basis.
\end{itemize}

\subsection{Block Matrix Structure of Equivariant Maps}
\label{subsec:block_structure_equivariant_operator}

\begin{theorem}[Schur's Lemma]\label{thm:lema_de_schur}
Let $\psi: G \to \GL(V)$ and $\phi: G \to \GL(W)$ be two irreducible representations of a group $G$, and let $\tau: V \to W$ be an equivariant map. Exactly one of the following is true:
\begin{itemize}
    \item[(i)] $\tau$ is the null linear transformation ($\tau = 0$).
    \item[(ii)] $\tau$ is a linear isomorphism, $\psi$ and $\phi$ are equivalent, and $\psi(g) = \tau \circ \phi(g) \circ \tau^{-1}$.
\end{itemize}
\end{theorem}

We can compare two representations by applying Schur's Lemma to each pair of irreducible constituents of the decomposition \eqref{eq:phi_decomposition}.

\begin{lemma}
\label{lem:lemma_block_zero}
Let $\psi: G \to \GL(V)$ and $\phi: G \to \GL(W)$ be representations with decompositions $\psi = b_1\psi_{1} \oplus \cdots \oplus b_k\psi_{k}$ and $\phi = c_1\phi_{1} \oplus \cdots \oplus c_n\phi_{n}$. Let $\tau: V \to W$ be an equivariant map. Let $\mathcal{T}_{ij}$ be the restriction of $\tau$ connecting $\psi_i$ and $\phi_j$.
\begin{itemize}
    \item[i)] $\mathcal{T}_{ij}$ vanishes whenever $\psi_i$ is not equivalent to $\phi_j$.
    \item[ii)] $\mathcal{T}_{ij}$ is decomposed as $\deg(\phi_j)$ identical copies of an operator $\mathfrak{t}_{ij}$ of order $b_i \times c_j$. In the corresponding matrix form relative to the symmetry-adapted basis, $\mathcal{T}_{ij}$ takes the block-diagonal form:
\begin{equation}
\label{eq:form_repeting_blocks}
\mathcal{T}_{ij} = \left (
\begin{array}{r|r|r}
\mathfrak{t}_{ij} &        &  \\\hline
       &\ddots  &  \\\hline
       &        &\mathfrak{t}_{ij} 
\end{array}
\right).
\end{equation}
\end{itemize}
\end{lemma}
\section{Formalism: DEC and Group Actions}
\label{sec:formalism_dec}

In this section, we rigorously define the algebraic and geometric foundations of Discrete Exterior Calculus (DEC) \cite{hirani2003,desbrun2005} in the presence of isometric group actions. Our formalism aligns with standard conventions in computational cohomology and representation theory.

\subsection{Simplicial complexes and discrete differential forms}

\begin{definition}[Simplicial Complex and Dual Complex]
Let $K$ be an oriented, $n$-dimensional simplicial complex embedded in $\mathbb{R}^N$. We denote the space of oriented $k$-chains over $\mathbb{R}$ by $C_k(K)$, with topological boundary operator $\partial_k: C_k(K) \to C_{k-1}(K)$. Let $\star K$ denote the circumcentric dual complex associated with $K$. For each primal $k$-simplex $\sigma \in K$, there is a corresponding uniquely determined dual $(n-k)$-cell denoted by $\star \sigma \in \star K$.
\end{definition}

\begin{definition}[Discrete Differential Forms]
The space of discrete $k$-forms on $K$, denoted $\Omega_d^k(K)$, is formally defined as the algebraic dual space of the $k$-chains:
\begin{equation}
    \Omega_d^k(K) := \Hom(C_k(K), \mathbb{R}).
\end{equation}
For any $\alpha \in \Omega_d^k(K)$ and $\sigma \in C_k(K)$, the natural pairing is given by the evaluation $\langle \alpha, \sigma \rangle := \alpha(\sigma)$. Geometrically, this bilinear pairing operates as the exact discrete analog of integrating a continuous exterior $k$-form over $\sigma$.
\end{definition}

\subsection{Induced group action on cochains}

Let $G$ be a finite group acting on the simplicial complex $K$. We require that the action of each element $g \in G$ acts as a simplicial isomorphism that strictly preserves the incidence relations of $K$, while additionally operating as a geometric isometry on the ambient space $\mathbb{R}^N$.

\begin{definition}[Induced Representation]
The isometric group action $g: K \to K$ induces a well-defined linear representation $\rho^k: G \to \GL(\Omega_d^k(K))$ on the space of discrete $k$-forms via the standard geometric pullback:
\begin{equation}
\label{eq:group_action}
    \langle \rho^k(g)\alpha, \sigma \rangle := \langle \alpha, g^{-1}\sigma \rangle, \quad \forall g \in G, \; \alpha \in \Omega_d^k(K), \; \sigma \in C_k(K).
\end{equation}
\end{definition}

\begin{remark}
The mapping $g^{-1}$ is linear, so it sends an oriented $k$-simplex $\sigma$ to an oriented $k$-simplex $g^{-1}\sigma \in K$. When $g^{-1}$ reverses the orientation of $\sigma$, the parity convention $\langle \alpha, -\sigma \rangle = -\langle \alpha, \sigma \rangle$ fixes the sign of the evaluation.
\end{remark}

\subsection{Fundamental DEC operators}

We now introduce the discrete differential operators of DEC, whose algebraic structure parallels the continuous de Rham complex.

\begin{definition}[Discrete Exterior Derivative]
The discrete exterior derivative $d_k: \Omega_d^k(K) \to \Omega_d^{k+1}(K)$ is defined by exact duality to the topological boundary operator $\partial_{k+1}$, thus satisfying a discrete generalized Stokes' theorem by construction:
\begin{equation}
\label{eq:stokes}
    \langle d_k \alpha, \sigma \rangle := \langle \alpha, \partial_{k+1} \sigma \rangle, \quad \forall \alpha \in \Omega_d^k(K), \; \sigma \in C_{k+1}(K).
\end{equation}
\end{definition}

\begin{definition}[Discrete Hodge Star]
\label{def:hodge}
The diagonal discrete Hodge star operator $\star_k: \Omega_d^k(K) \to \Omega_d^{n-k}(\star K)$ establishes an isomorphism coupling the primal $k$-cochains to the dual $(n-k)$-cochains according to local metric geometries:
\begin{equation}
\label{eq:hodge}
    \langle \star_k \alpha, \star \sigma \rangle := \frac{|\star \sigma|}{|\sigma|} \langle \alpha, \sigma \rangle,
\end{equation}
where $|\sigma|$ denotes the standard $k$-dimensional Euclidean volume of the primal simplex $\sigma$, and $|\star\sigma|$ denotes the $(n-k)$-dimensional volume of its respective circumcentric dual cell.
\end{definition}

\begin{remark}[Cotangent formula as volume ratio]
Within the circumcentric DEC framework, the abstract ratio $|\star e|/|e|$ in~\eqref{eq:hodge} admits an explicit trigonometric evaluation for 1-forms. Because the dual vertex is placed at the circumcentre of each primal triangle, the dual edge $\star e$ is orthogonal to $e$, and elementary trigonometry yields
\[
  \star_1(e) \;=\; \frac{|\star e|}{|e|} \;=\; \frac{\cot\alpha_e + \cot\beta_e}{2},
\]
where $\alpha_e$ and $\beta_e$ are the angles opposite to edge $e$ in its two adjacent triangles~\cite{desbrun2005}. Strict positivity of $\star_1(e)$ is equivalent to the dual edge having positive length, which fails precisely when $\alpha_e + \beta_e \geq \pi$.
\end{remark}

\begin{definition}[Discrete Codifferential and Laplacian]
The discrete formulation of the codifferential mapping $\delta_k: \Omega_d^k(K) \to \Omega_d^{k-1}(K)$ is implicitly formalized as:
\begin{equation}
\label{eq:codifferential}
    \delta_k := (-1)^{n(k-1)+1}\, \star_{k-1}^{-1} \circ d_{n-k} \circ \star_k.
\end{equation}
Here the sign $(-1)^{n(k-1)+1}$ follows the convention of~\cite{hirani2003}, where the orientation of the dual complex is chosen so that the primal--dual pairing is positive. With this convention $\delta_k$ is the formal adjoint of $d_{k-1}$ with respect to the $L^2$ inner product on cochains.
Consequently, the discrete Hodge Laplacian acting algebraically on $k$-forms is synthesized via the canonical decomposition:
\begin{equation}
\label{eq:laplacian}
    \Delta_k := \delta_{k+1} \circ d_k + d_{k-1} \circ \delta_k.
\end{equation}
\end{definition}

\section{Equivariance of DEC Operators}
\label{sec:equivariance}

The central result is that the fundamental DEC operators are equivariant, and hence are block-diagonal in the symmetry-adapted bases derived in Section~\ref{sec:finding_symmetry_adapted_basis}

\subsection{Equivariance of the discrete exterior derivative}

\begin{theorem}
\label{thm:d_equiv}
Assume the finite group $G$ transforms the complex $K$ via strictly simplicial isomorphisms. The discrete exterior derivative $d_k: \Omega_d^k(K) \to \Omega_d^{k+1}(K)$ constitutes an equivariant operator satisfying the equivariant relation:
\begin{equation}
    d_k \circ \rho^k(g) = \rho^{k+1}(g) \circ d_k, \quad \forall g \in G.
\end{equation}
\end{theorem}

\begin{proof}
Let $\alpha \in \Omega_d^k(K)$, let $\sigma \in C_{k+1}(K)$, and let $g \in G$. Combining the discrete Stokes' theorem~\eqref{eq:stokes} with the pullback definition~\eqref{eq:group_action} gives:
\begin{align}
    \langle d_k(\rho^k(g)\alpha), \sigma \rangle &= \langle \rho^k(g)\alpha, \partial_{k+1}\sigma \rangle \notag \\
    &= \langle \alpha, g^{-1}(\partial_{k+1}\sigma) \rangle. \label{eq:d_step1}
\end{align}
Because $g$ is a simplicial isomorphism, it preserves the incidence relations of the complex, so the group action commutes with the topological boundary operator:
\begin{equation}
\label{eq:boundary_commute}
    g^{-1}(\partial_{k+1}\sigma) = \partial_{k+1}(g^{-1}\sigma).
\end{equation}
Substituting \eqref{eq:boundary_commute} into \eqref{eq:d_step1}, and using \eqref{eq:stokes}:
\begin{align}
    \langle \alpha, \partial_{k+1}(g^{-1}\sigma) \rangle &= \langle d_k\alpha, g^{-1}\sigma \rangle \notag \\
    &= \langle \rho^{k+1}(g)(d_k\alpha), \sigma \rangle.
\end{align}
Since this holds for every $(k+1)$-simplex $\sigma$, we conclude that $d_k \circ \rho^k(g) = \rho^{k+1}(g) \circ d_k$.
\end{proof}

\begin{remark}
The equivariance of $d_k$ is purely topological and metric-agnostic, requiring only that $G$ act by simplicial automorphisms; it is the discrete counterpart of the naturality of the exterior derivative under smooth pullback. In particular, the result holds for $n$-complexes of arbitrary dimension, since the proof makes no use of the topological dimension of $K$. This generality is exploited in Section~\ref{sec:exp_3complex}.
\end{remark}

\subsection{Equivariance of the discrete Hodge star}

Unlike the discrete exterior derivative, which is purely topological, the discrete Hodge star depends on the metric through the volume ratios of Definition~\ref{def:hodge}. Its equivariance therefore needs more than the combinatorial action of $G$: the group must act by isometries, so that these volumes are preserved along each orbit.

\begin{theorem}
\label{thm:star_equiv}
Suppose the finite group $G$ acts on the simplicial complex $K$ by exact isometries of $\mathbb{R}^N$. Then:
\begin{enumerate}
    \item[(i)] The circumcentric dual $\star K$ is invariant under the action of $G$, in the sense that $g(\star\sigma) = \star(g\sigma)$ for every simplex $\sigma \in K$.
    \item[(ii)] The discrete Hodge star $\star_k: \Omega_d^k(K) \to \Omega_d^{n-k}(\star K)$ is equivariant:
    \begin{equation}
        \star_k \circ \rho^k(g) = \rho^{n-k}(g) \circ \star_k, \quad \forall g \in G.
    \end{equation}
\end{enumerate}
\end{theorem}

\begin{proof}
\textit{Proof of (i): invariance of the dual complex.} Let $\sigma = [v_0, \ldots, v_k]$ be a $k$-simplex with circumcenter $c_\sigma$, the unique point of the affine span of $\sigma$ equidistant from its vertices. Because $g$ acts as an isometry, it preserves distances and affine spans, so it carries the circumcenter of $\sigma$ to that of $g\sigma$:
\begin{equation}
    g(c_\sigma) = c_{g\sigma}.
\end{equation}
The dual cell $\star\sigma$ is assembled from the circumcenters of the cofaces $\tau \supseteq \sigma$. Since $g$ preserves the face relation, $\tau \supseteq \sigma$ implies $g\tau \supseteq g\sigma$, and the identity above sends each circumcenter generating $\star\sigma$ to the corresponding circumcenter of $\star(g\sigma)$. Therefore $g(\star\sigma) = \star(g\sigma)$.

\vspace{0.5em}
\noindent \textit{Proof of (ii): equivariance of $\star_k$.} We evaluate the two compositions $\star_k \circ \rho^k(g)$ and $\rho^{n-k}(g) \circ \star_k$ on an arbitrary cochain $\alpha \in \Omega_d^k(K)$ and compare their values on each dual cell $\star\sigma$. For the first composition, the Hodge-star definition~\eqref{eq:hodge} and the action of $\rho^k(g)$ give
\begin{align}
    \langle \star_k(\rho^k(g)\alpha), \star\sigma \rangle &= \frac{|\star\sigma|}{|\sigma|} \langle \rho^k(g)\alpha, \sigma \rangle \notag \\
    &= \frac{|\star\sigma|}{|\sigma|} \langle \alpha, g^{-1}\sigma \rangle. \label{eq:star_step1}
\end{align}
Because $g$ is an isometry, it preserves both the primal $k$-volumes and the dual $(n-k)$-volumes; together with part~(i) this gives
\begin{equation}
\label{eq:vol_invariance}
    |g^{-1}\sigma| = |\sigma| \quad \text{and} \quad |\star(g^{-1}\sigma)| = |g^{-1}(\star\sigma)| = |\star\sigma|,
\end{equation}
so the volume ratio that defines the Hodge star is constant along the orbit of $\sigma$:
\begin{equation}
\label{eq:ratio_invariance}
    \frac{|\star(g^{-1}\sigma)|}{|g^{-1}\sigma|} = \frac{|\star\sigma|}{|\sigma|}.
\end{equation}
For the second composition, the same definition and the action of $\rho^{n-k}(g)$ on the dual cell yield
\begin{align}
    \langle \rho^{n-k}(g)(\star_k\alpha), \star\sigma \rangle &= \langle \star_k\alpha, g^{-1}(\star\sigma) \rangle \notag \\
    &= \langle \star_k\alpha, \star(g^{-1}\sigma) \rangle \notag \\
    &= \frac{|\star(g^{-1}\sigma)|}{|g^{-1}\sigma|} \langle \alpha, g^{-1}\sigma \rangle \notag \\
    &= \frac{|\star\sigma|}{|\sigma|} \langle \alpha, g^{-1}\sigma \rangle, \label{eq:star_step2}
\end{align}
where the second equality uses part~(i) and the last uses~\eqref{eq:ratio_invariance}. The right-hand sides of~\eqref{eq:star_step1} and~\eqref{eq:star_step2} agree for every $\sigma$, so $\star_k \circ \rho^k(g) = \rho^{n-k}(g) \circ \star_k$.
\end{proof}

\begin{corollary}[Dual-agnostic equivariance]
\label{cor:dual_agnostic}
Let $G$ act on $K$ by isometries and let $\sigma \mapsto \star\sigma$ be any dual assignment that is $G$-equivariant, in the sense that $g(\star\sigma) = \star(g\sigma)$ for every $\sigma \in K$ and $g \in G$. Then the discrete Hodge star $\star_k$ defined by the volume ratios of Definition~\ref{def:hodge} is equivariant, $\star_k \circ \rho^k(g) = \rho^{n-k}(g) \circ \star_k$. The circumcentric dual is the instance covered by Theorem~\ref{thm:star_equiv}; the barycentric (centroid) dual used in Section~\ref{sec:exp_3complex} is another, since an isometry sends the barycenter of a coface to the barycenter of its image, $g(b_\tau) = b_{g\tau}$.
\end{corollary}

\begin{proof}
The proof of Theorem~\ref{thm:star_equiv}(ii) uses the dual complex only through the equivariance $g(\star\sigma) = \star(g\sigma)$ and the isometry-invariance of the dual volumes $|\star\sigma|$, the latter automatic because $G$ acts by isometries. Both hold under the stated hypothesis, so the argument of that proof applies verbatim.
\end{proof}

\begin{remark}
\label{rmk:wellcentered}
The equivariance of Theorem~\ref{thm:star_equiv} uses only that $G$ acts by isometries, so well-centeredness is not among its hypotheses. Well-centeredness enters separately, as the condition under which every primal simplex contains its circumcenter in its interior. It makes the dual cells of~\cite{hirani2003} well defined and the Hodge volumes $|\!\star\sigma|/|\sigma|$ strictly positive, the positivity that renders each $\star_k$ invertible in Lemma~\ref{lem:delta_equiv}. In two dimensions, well-centeredness implies the local Delaunay condition but is strictly stronger, since a Delaunay triangulation may contain obtuse triangles. For the circumcentric dual, the weaker Delaunay property already renders these volumes positive, strictly so except in the cocircular degenerate case where opposite angles sum to exactly $\pi$~\cite{bobenko2006}. In three dimensions the condition is more restrictive, because Delaunay tetrahedralizations need not be well-centered~\cite{vanderzee2010,hirani2013} and the circumcentric dual may then acquire negative dual volumes. Section~\ref{sec:exp_3complex} addresses this gap by replacing the circumcentric dual with a centroid-based Hodge star on the \emph{fan-BCC mesh}, a $T_d$-symmetric tetrahedralization of $T^3$ in which each unit cube is split into twelve tetrahedra fanning from its body centre. This construction is not well-centered~\cite{vanderzee2010}, yet Corollary~\ref{cor:dual_agnostic} shows that the equivariance of Theorem~\ref{thm:star_equiv} continues to hold for it, because the barycentric dual assignment is itself $G$-equivariant.
\end{remark}

\begin{remark}
\label{rmk:averaging}
A mesh built in floating-point arithmetic realizes the symmetry group $G$ only approximately. A group element fixes the connectivity exactly, but the embedded image $g\sigma$ of a simplex differs from $\sigma$ by rounding-level perturbations. The volume ratios $|\star\sigma|/|\sigma|$ that define the discrete Hodge star (Definition~\ref{def:hodge}) are then not exactly equal across a $G$-orbit, and the isometry hypothesis of Theorem~\ref{thm:star_equiv} holds only up to that error. Averaging each ratio over the stabilizer $G_\sigma$ of $\sigma$,
\begin{equation}
    \overline{r}(\sigma) := \frac{1}{|G_\sigma|}\sum_{g \in G_\sigma} \frac{|\star(g\sigma)|}{|g\sigma|},
\end{equation}
replaces the perturbed local weights by a single symmetrized value and removes the discrepancy. The orbit-averaging construction is independent of the topological dimension of $K$ and of the choice of dual cell (circumcentric or centroid-based); it restores the equivariance relation whenever the symmetry $G$ acts only approximately at the floating-point level, at the cost of a one-time geometric preprocessing step.
\end{remark}

\subsection{Equivariance of the codifferential}

Equivariance is preserved under the algebraic operations from which the codifferential is assembled. Because $\delta_k$ is, up to an overall sign, the composition of an exterior derivative with two Hodge stars (one of them entering through its inverse), its equivariance follows from Theorems~\ref{thm:d_equiv} and~\ref{thm:star_equiv} once we observe that the inverse of an equivariant isomorphism is again equivariant.

\begin{lemma}
\label{lem:delta_equiv}
Let $G$ act on $K$ by isometries, and suppose each Hodge star $\star_j$ is invertible. Then the discrete codifferential $\delta_k:\Omega_d^k(K)\to\Omega_d^{k-1}(K)$ of~\eqref{eq:codifferential} is equivariant:
\begin{equation}
    \delta_k \circ \rho^k(g) = \rho^{k-1}(g) \circ \delta_k, \quad \forall g \in G.
\end{equation}
\end{lemma}

\begin{proof}
We first record that the inverse of an equivariant isomorphism is equivariant. By Theorem~\ref{thm:star_equiv} the star $\star_{k-1}$ intertwines $\rho^{k-1}$ and $\rho^{n-k+1}$,
\begin{equation*}
    \star_{k-1}\circ\rho^{k-1}(g) = \rho^{n-k+1}(g)\circ\star_{k-1};
\end{equation*}
composing on both sides with $\star_{k-1}^{-1}$ rearranges this to
\begin{equation*}
    \rho^{k-1}(g)\circ\star_{k-1}^{-1} = \star_{k-1}^{-1}\circ\rho^{n-k+1}(g),
\end{equation*}
so $\star_{k-1}^{-1}$ intertwines $\rho^{n-k+1}$ and $\rho^{k-1}$. The codifferential~\eqref{eq:codifferential} is the composition
\begin{equation*}
    \Omega_d^k(K) \xrightarrow{\;\star_k\;} \Omega_d^{n-k}(K) \xrightarrow{\;d_{n-k}\;} \Omega_d^{n-k+1}(K) \xrightarrow{\;\star_{k-1}^{-1}\;} \Omega_d^{k-1}(K),
\end{equation*}
scaled by the constant $(-1)^{n(k-1)+1}$. Its three factors intertwine, respectively, $\rho^k\!\to\!\rho^{n-k}$ (Theorem~\ref{thm:star_equiv}), $\rho^{n-k}\!\to\!\rho^{n-k+1}$ (Theorem~\ref{thm:d_equiv}), and $\rho^{n-k+1}\!\to\!\rho^{k-1}$ (the inversion step above). Chaining the three intertwining relations and absorbing the scalar, which commutes with every $\rho^{k-1}(g)$, yields $\delta_k\circ\rho^k(g) = \rho^{k-1}(g)\circ\delta_k$ for all $g\in G$.
\end{proof}

\begin{remark}
The hypothesis of Lemma~\ref{lem:delta_equiv} is invertibility of the participating Hodge stars, which is what the proof uses; well-centeredness (Remark~\ref{rmk:wellcentered}) is a sufficient geometric condition for it, not a necessary one. The cotangent positivity of $\star_1$ in two dimensions~\cite{bobenko2006} and the centroid-based Hodge of Section~\ref{sec:exp_3complex}, which is not well-centered, both meet this requirement by construction.
\end{remark}

\section{Universal Block-Diagonalization}
\label{sec:universal_block}

We now state the main structural result, which captures the \emph{universality} of the DEC approach to symmetry-adapted computation.

\begin{theorem}[Universal Block-Diagonalization]
\label{thm:universal}
Under the hypotheses of Theorems~\ref{thm:d_equiv} and~\ref{thm:star_equiv}, let $G$ be a finite group acting on a simplicial complex $K$ by isometries, and let $d_k$ and $\star_k$ denote the DEC exterior derivative and Hodge star, respectively. Suppose in addition that each participating Hodge star $\star_j$ is invertible. Then:
\begin{enumerate}
    \item[(i)] Any linear operator $P: \Omega_d^k(K) \to \Omega_d^\ell(K)$ that can be expressed as a composition and/or linear combination of the operators $\{d_j, \star_j, \star_j^{-1}\}$ and scalar multiplication is equivariant under $G$.
    \item[(ii)] Let $\{\phi_1, \ldots, \phi_r\}$ be the distinct irreducible representations of $G$, with $\deg \phi_i = n_i$. For each $\phi_i$, define the isotypic projection operator on $\Omega_d^k(K)$ (analogous to equation \eqref{eq:isotypicDecomposition}):
    \begin{equation}
    \label{eq:projection}
        \Pi_i^k = \frac{n_i}{|G|}\sum_{g \in G} \chi_i(g^{-1})\,\rho^k(g),
    \end{equation}
    where $\chi_i$ is the character of $\phi_i$. Then $\Omega_d^k(K)$ decomposes uniquely into isotypic components:
    \begin{equation}
        \Omega_d^k(K) = \bigoplus_{i=1}^r V_i^k, \qquad V_i^k = \Pi_i^k\,\Omega_d^k(K).
    \end{equation}
    \item[(iii)] When both domain and codomain are expressed in their respective symmetry-adapted bases, constructed via the transfer operators \eqref{eq:transferences} and ordered canonically by irreducible representation (Section~\ref{subsec:algorithm}), the matrix of $P$ strictly assumes a block-diagonal form:
    \begin{equation}
        [P]_{\text{sym}} = \bigoplus_{i=1}^r [P]_i,
    \end{equation}
    where $[P]_i$ is the exact block matrix corresponding to the $i$-th irreducible representation.
\end{enumerate}
\end{theorem}

\begin{proof}
Part~(i) follows from the fact that equivariant operators are closed under composition, linear combination, and scalar multiplication, and that the inverse of an equivariant isomorphism is again equivariant, as established in the proof of Lemma~\ref{lem:delta_equiv}. Since $d_k$ (Theorem~\ref{thm:d_equiv}) and $\star_k$ (Theorem~\ref{thm:star_equiv}) are equivariant and each $\star_j$ is invertible by hypothesis, any algebraic expression built from $\{d_j, \star_j, \star_j^{-1}\}$ is equivariant.

Part~(ii) follows directly from the character properties reviewed in Section~\ref{sec:finding_symmetry_adapted_basis}: the operators $\Pi_i^k$ form a complete set of orthogonal idempotents summing to the identity, and each $V_i^k$ expands as a direct sum of equivalent copies of $\phi_i$.

Part~(iii) translates the global constraint of Schur's Lemma (Theorem~\ref{thm:lema_de_schur}) onto the adapted bases. An equivariant operator restricts within matched isotypic components, forcing the cross-component blocks to vanish (Lemma~\ref{lem:lemma_block_zero}).
\end{proof}

\begin{remark}
The symmetry-adapted basis for the cochain space $\Omega_d^k(K)$ is determined solely by the group action on $K$ and is independent of the specific PDE being solved. Once computed, it simultaneously block-diagonalizes:
\begin{itemize}
    \item the scalar Laplacian $\Delta_0 = \delta_1 d_0$,
    \item the vector Laplacian $\Delta_1 = \delta_2 d_1 + d_0 \delta_1$,
    \item Maxwell-type operators,
    \item elasticity operators,
    \item any other operator assembled from DEC building blocks.
\end{itemize}
This contrasts with approaches (such as \cite{olver2025}) where equivariance must be verified and the symmetry-adapted decomposition must be constructed from scratch for each specific differential operator.
\end{remark}

\begin{remark}[Anisotropic data and boundary conditions]
\label{rmk:anisotropic_data}
The block-diagonalization of Theorem~\ref{thm:universal} is a property of the \emph{operator}, not of the data~\cite{bossavit1993}. When the mesh geometry and operator coefficients are $G$-equivariant, the system matrix is block-diagonal in the symmetry-adapted basis irrespective of the symmetry of the right-hand side or the boundary conditions. Inhomogeneous Dirichlet conditions enter through a lifting function and inhomogeneous Neumann conditions through the load vector; the load vector, however anisotropic, is projected onto the isotypic blocks, which are then solved independently. The construction mirrors the spectral setting of~\cite[Sec.~5]{olver2025}, where the operator coefficients are group-invariant while the load vector and the solution carry no symmetry.
\end{remark}

\begin{remark}[Nonlinear equations and operator splitting]
\label{rmk:nonlinear}
Theorem~\ref{thm:universal} applies to linear operators assembled from $d$ and~$\star$. For nonlinear PDEs such as the Navier--Stokes equations, the framework remains applicable within operator-splitting (fractional-step) time integrators~\cite{guermond2006}. The dominant implicit substeps, typically a viscous diffusion solve (Helmholtz equation) and a pressure-projection solve (Poisson equation), are linear and, provided the viscosity and geometry respect~$G$, equivariant. Assembled from $d$ and~$\star$~\cite{elcott2007}, they are block-diagonalized by the precomputed matrices~$Q_k$ at every time step, whereas the nonlinear terms depend on the generally asymmetric solution field and are evaluated in the physical basis. The method isolates rather than eliminates the nonlinear cost; the implicit solves are accelerated while the typically subdominant explicit evaluation is unchanged. The same holds for any time integrator that treats stiff linear terms implicitly and nonlinear terms explicitly, provided the implicit part is assembled from equivariant DEC operators~\cite[Sec.~7]{olver2025}.
\end{remark}

\subsection{Algorithmic construction of symmetry-adapted bases}
\label{subsec:algorithm}

The practical construction
of the symmetry-adapted basis $Q_k$ for meshes with thousands of simplices requires
a numerically stable procedure.
General-purpose software frameworks such as \texttt{PySymmetry}~\cite{pysymmetry2025}
can facilitate this task. Given a finite group $G$, they automatically enumerate
its conjugacy classes, compute the full character table, and determine all irreducible
representations, thereby providing the exact algebraic input required by the projection
formula~\eqref{eq:projection} without manual derivation.
The results reported in this paper were produced by a lightweight numerical extraction
pipeline, developed for this work and tailored to the point groups of our target curved manifolds (the icosahedral
group $I_h$ for spherical meshes and the point group $D_{6h}$ for toroidal grids),
which applies the character values to the large sparse representation matrices of the
mesh and extracts the symmetry-adapted basis via spectral decomposition and QR factorization.
We now describe this pipeline, which realizes the block-diagonal structure of
Theorem~\ref{thm:universal} for any finite symmetry group $G$.

\begin{itemize}

\vspace{0.15em}

\item \noindent\textit{Step 1: Group enumeration.}
Given a set of geometric generators of $G$ (e.g., rotations about symmetry axes,
reflections, or inversions), all $|G|$ group elements are enumerated by breadth-first search (BFS).
Each element $g$ is stored as a pair $(\pi_v, M)$, where $M \in O(N)$ is the $N \times N$
orthogonal matrix acting on $\mathbb{R}^N$ and $\pi_v \in S_{N_0}$ is the induced permutation
of the vertex set.  The permutation array $\pi_v$ (an integer vector) serves as the
unique identifier for $g$, avoiding floating-point comparisons entirely.

\vspace{0.15em}

\item \noindent\textit{Step 2: Representation matrices on $k$-cochains.}
For each group element $g \in G$, the linear representation
$\rho^k(g) \in \GL(\Omega_d^k(K))$ defined by~\eqref{eq:group_action}
is assembled as a signed permutation matrix.
For $0$-forms (vertices), $\rho^0(g)$ is the standard permutation matrix induced by $\pi_v$.
For $k$-forms with $k \geq 1$, the action on an oriented $k$-simplex $\sigma = [v_0, \ldots, v_k]$
maps it to $g\sigma = [g(v_0), \ldots, g(v_k)]$; if the resulting simplex requires
reordering to match the canonical orientation, the matrix entry acquires a sign factor
$\sgn(g,\sigma) = \pm 1$.  Thus $\rho^k(g)$ has exactly $N_k$ nonzero entries, each equal to $\pm 1$.

\vspace{0.15em}

\item \noindent\textit{Step 3: Isotypic projection.}
For each irreducible representation $\phi_i$ with character $\chi_i$ and degree $n_i$,
the isotypic projector $\Pi_i^k$ onto the component $V_i^k$ is computed via the
character projection formula~\eqref{eq:projection}:
\[
  \Pi_i^k \;=\; \frac{n_i}{|G|} \sum_{g \in G} \chi_i(g^{-1})\,\rho^k(g).
\]
Since every $\rho^k(g)$ is a signed permutation (with $N_k$ nonzero entries),
accumulating the sum costs $\mathcal{O}(|G| \cdot N_k)$ arithmetic operations per irrep (,
yielding a dense symmetric matrix $\Pi_i^k \in \mathbb{R}^{N_k \times N_k}$
whose eigenvalues lie in $\{0,1\}$.
In practice, since characters are class functions, the sum over $|G|$ elements
collapses to a sum over the $N_{\mathrm{cl}}$ conjugacy classes of $G$,
reducing the cost to $\mathcal{O}(N_{\mathrm{cl}} \cdot N_k)$ per irrep.

\vspace{0.15em}

\item \noindent\textit{Step 4: Orthonormal basis extraction.}
An orthonormal basis for each $V_i^k = \operatorname{Im}(\Pi_i^k)$ is obtained by
spectral decomposition of the symmetric projector $\Pi_i^k$: retaining the eigenvectors
associated with the unit eigenvalue yields a basis of dimension
$\dim V_i^k = c_i \cdot n_i$, where $c_i$ is the multiplicity computed via
the character inner product~\eqref{eq:Produto_interno_dos_Carateres}.
A subsequent QR factorization restores orthonormality to machine precision, correcting floating-point round-off from the eigendecomposition. In exact arithmetic the isotypic construction is already orthonormal, so this QR step is a finite-precision safeguard rather than an algebraic requirement.
This step dominates the overall cost at $\mathcal{O}(N_k^3)$
(via symmetric eigendecomposition), but is performed \emph{once} as a preprocessing
step and amortized over all operators and right-hand sides.

\vspace{0.15em}

\item \noindent\textit{Step 5: Assembly of the change-of-basis matrix.}
The global orthogonal matrix $Q_k$ is formed by concatenating the orthonormal
bases in the canonical order of the irreducible representations:
\[
  Q_k \;=\; \bigl[\, \tilde{Q}_1^k \;\big|\; \tilde{Q}_2^k \;\big|\; \cdots \;\big|\; \tilde{Q}_r^k \,\bigr],
\]
where $\tilde{Q}_i^k \in \mathbb{R}^{N_k \times \dim V_i^k}$ and $r$ is the number of
distinct irreps.  By construction $Q_k^\top Q_k = I_{N_k}$, so the change of basis
is orthogonal and numerically stable.  The column ordering assurres the order of symmetry adapted basis.
For any equivariant operator $P$, the transformed matrix
$Q_k^\top P\, Q_k$ is block-diagonal with blocks indexed by the irreps of $G$.
\end{itemize}
\paragraph{Group-specific considerations.}
While the algorithm above applies to any finite symmetry group, the choice
of generators and the classification of group elements into conjugacy
classes are group-dependent.
For the icosahedral group $I_h$ (order $120$, $10$ irreps of dimensions
$1,3,3,4,5$ in gerade/ungerade pairs), the generators are a $72^\circ$ rotation
about a $5$-fold axis, a $120^\circ$ rotation about a $3$-fold axis, and the
spatial inversion.
For the point group $D_{6h}$ (order $24$, $12$ irreps of dimensions $1$ or $2$),
the generators are a $60^\circ$ rotation about the principal axis, a $180^\circ$
rotation about a perpendicular axis, and the horizontal mirror reflection.
For the tetrahedral point group $T_d$ (order $24$, $5$ irreps of
dimensions $1,1,2,3,3$, namely $A_1,A_2,E,T_1,T_2$), the
group-enumeration step (Step~1 of the algorithm in
Section~\ref{subsec:algorithm}) admits a combinatorial shortcut via the
isomorphism $T_d \cong S_4$: the $24$ elements are enumerated directly as
signed permutations of $(x,y,z)$ with product of signs equal to $+1$,
bypassing the BFS but producing the same set of orthogonal matrices
required by Steps~2--5.
In all three cases, conjugacy-class membership is determined from the trace and
determinant of the orthogonal matrix $M$, which suffices to evaluate the characters
$\chi_i(g)$ appearing in the projection formula.
All characters of $I_h$, $D_{6h}$ and $T_d$ are real-valued, so $\chi_i(g^{-1}) = \chi_i(g)$,
and the projectors $\Pi_i^k$ are symmetric.

\subsection{Computational complexity reduction}
\label{subsec:complexity}

Let $N_k = \dim \Omega_d^k(K)$ be the total number of $k$-simplices. Without symmetry adaptation, a dense direct factorization of $\Delta_k$ costs $\Theta(N_k^3)$ operations. Because $\Delta_k$ carries only $O(N_k)$ nonzeros, a sparse direct or iterative solver, the baseline a practitioner would use, can reduce this substantially; the FLOP counts reported here are stated against the dense factorization, the like-for-like comparator for the dense isotypic blocks. In the symmetry-adapted basis, the dense factorization cost reduces to
\begin{equation}
    \sum_{i=1}^r \Theta\bigl(\dim(V_i^k)^3\bigr).
\end{equation}
Since $\sum_i \dim V_i^k = N_k$ and the cost is dominated by the largest block, the reduction is greatest when the group $G$ splits the space into many components of comparable dimension. Moreover, the $r$ independent blocks can be factorized in parallel, yielding an embarrassingly parallel algorithm. The block-diagonalization is orthogonal to sparsity. It decouples the system by symmetry regardless of how each block is subsequently factored, and therefore composes with a sparse per-block solver rather than competing with it.

For example, consider the full icosahedral group $I_h$ with $|G| = 120$ and 10 irreducible representations of dimensions $d_\mu \in \{1, 3, 3, 4, 5\}$ (gerade and ungerade, satisfying $\sum d_\mu^2 = 120$). For the 1-form Laplacian $\Delta_1$ on the geodesic icosphere at subdivision level $n=4$, the cochain space has dimension $N_1 = 480$ and the ten isotypic blocks have dimensions $4, 36, 36, 64, 100, 4, 36, 36, 64, 100$. The sequential direct-solve cost reduces by a factor of
\begin{equation}
    \frac{N_1^3}{\sum_\mu (\dim V_\mu^1)^3} = \frac{480^3}{2(4^3 + 36^3 + 36^3 + 64^3 + 100^3)} \approx 41.
\end{equation}
When the ten blocks are factorized in parallel, the wall-clock time is dominated by the largest block (dimension $100$), yielding a parallel speedup of $480^3 / 100^3 \approx 110$. This worked example uses the 1-form Laplacian $\Delta_1$. The speedups measured in Section~\ref{sec:experiments} (Tables~\ref{tab:performance_sphere}--\ref{tab:performance_torus}) are instead for the 0-form Laplacian $\Delta_0$, whose smaller cochain space and different block distribution give the distinct ratios reported there.

This isotypic reduction is the first of two levels. By Schur's lemma the equivariant operator restricted to each isotypic component $V_\mu^k$ acts as $M_\mu \otimes I_{d_\mu}$, so the block of dimension $c_\mu d_\mu$ is $d_\mu$ identical copies of a single $c_\mu \times c_\mu$ matrix $M_\mu$, with $c_\mu = \dim V_\mu^k / d_\mu$ the multiplicity. Factorizing one copy per irrep replaces $\sum_\mu (\dim V_\mu^k)^3$ by $\sum_\mu c_\mu^3$; for the same $\Delta_1$ example the sequential factor becomes $480^3 / [\,2(4^3 + 12^3 + 12^3 + 16^3 + 20^3)\,] \approx 3541$. In general the sequential and parallel asymptotes rise from $|G|^3/\sum_\rho d_\rho^6$ and $|G|^3/d_{\max}^6$ to $|G|^3/\sum_\rho d_\rho^3$ and $|G|^3/d_{\max}^3$, a further factor of order $|G|$. This deeper reduction is standard in computational group theory; we record it here because it applies verbatim to the DEC operators of Theorem~\ref{thm:universal}.

\section{Numerical Experiments}
\label{sec:experiments}

We present a sequence of numerical experiments that validate the theoretical framework and demonstrate its computational advantages. Experiments~1--3 are carried out on two curved test surfaces chosen to span different topologies, curvature profiles, and symmetry groups: a geodesic icosahedral sphere ($S^2$, genus~0, point group $I_h$ of order 120 with 10 irreps) and a torus of revolution ($T^2$, genus~1, point group $D_{6h}$ of order 24 with 12 irreps). Experiment~4 (Section~\ref{sec:exp_3complex}) extends the verification to a three-dimensional tessellation, the flat 3-torus $T^3$ under the tetrahedral group $T_d$. We begin by documenting the verification tests performed on the discrete implementation.

\subsection{Verification of the discrete implementation}
\label{sec:verification}

The verification tests in this subsection are performed on a geodesic icosphere at subdivision level $n=3$, consisting of $N_0 = 92$ vertices, $N_1 = 270$ edges, and $N_2 = 180$ faces. The implementation comprises three independent modules: mesh generation, DEC operator assembly, and symmetry-adapted basis construction. Each is verified by a dedicated test suite, with results summarized in Table~\ref{tab:verification}.

Standard topological and operator identities ($\chi = 2$, $\partial_1\partial_2 = 0$, primal--dual orthogonality, self-adjointness of $\star_k\Delta_k$, de~Rham cohomology dimensions $\dim\ker\Delta_0 = 1$ and $\dim\ker\Delta_1 = 0$, basis orthogonality $\|Q_k^\top Q_k - I\| < 2\times 10^{-15}$, and completeness $\sum_\mu \dim V_\mu^k = N_k$) all hold to machine precision and serve as software self-checks. For an operator $A:\Omega_d^k(K)\to\Omega_d^m(K)$ we write $[\rho(g),A] := \rho^m(g)\,A - A\,\rho^k(g)$ for its equivariance defect, which reduces to an ordinary commutator when $m=k$; equivariance of $A$ is the vanishing of this defect for all $g\in G$. The load-bearing residuals are summarized in Table~\ref{tab:verification}: equivariance of $d_0$ and $\star_1$ under all 120 elements of $I_h$ holds at machine precision (Theorems~\ref{thm:d_equiv}, \ref{thm:star_equiv}), and the off-block-diagonal residuals of $\Delta_0$ and $\Delta_1$ in the symmetry-adapted basis are below $10^{-15}$, confirming exact block-diagonalization (Theorem~\ref{thm:universal}).

A stronger universality test goes beyond block-diagonalizing the operators used to build $Q_0$. Stripping the Hodge stars from $\Delta_0 = \star_0^{-1} d_0^\top \star_1 d_0$ yields the metric-free Laplacian $d_0^\top d_0$, a structurally distinct operator that does not appear anywhere in the construction of $Q_0$; nevertheless, the same $Q_0$ block-diagonalizes it with residual $6.3\times 10^{-16}$. This confirms that the symmetry-adapted basis is a property of the group action on the mesh, not of any particular operator built from it.

\begin{table}[h]
\centering
\caption{Load-bearing verification residuals for the geodesic icosphere at $n=3$ ($N_0=92$, $N_1=270$, $N_2=180$, $|I_h|=120$). Standard topological and operator identities (Euler characteristic, boundary identity, primal--dual orthogonality, self-adjointness, de~Rham cohomology, basis orthogonality and completeness) all hold to machine precision $\varepsilon_{\mathrm{mach}} \approx 2.2\times10^{-16}$ and are omitted here.}
\label{tab:verification}
\begin{tabular}{llc}
\toprule
\textbf{Test} & \textbf{Quantity} & \textbf{Residual} \\
\midrule
Equivariance of $d_0$ under $I_h$         & $\max_g \|[\rho^1(g), d_0]\|$            & $0$ \\
Equivariance of $\star_1$ under $I_h$     & $\max_g \|[\rho^1(g), \star_1]\|$        & $5.9\times 10^{-15}$ \\
Off-block residual of $\Delta_0$          & $\|Q_0^\top \Delta_0 Q_0 - \mathrm{blockdiag}\|$ & $6.6\times 10^{-16}$ \\
Off-block residual of $\Delta_1$          & $\|Q_1^\top \Delta_1 Q_1 - \mathrm{blockdiag}\|$ & $9.6\times 10^{-16}$ \\
Universality ($d_0^\top d_0$, no Hodge)   & $\|Q_0^\top (d_0^\top d_0) Q_0 - \mathrm{blockdiag}\|$ & $6.3\times 10^{-16}$ \\
\bottomrule
\end{tabular}
\end{table}

\subsection{Experiment 1: Block-diagonalization of DEC building blocks}

This experiment verifies, on two geometrically distinct surfaces, that the DEC building blocks are individually block-diagonal in the symmetry-adapted basis and that the composite Hodge Laplacian $\Delta_1$ inherits this structure by algebraic composition.

\paragraph{Icosahedral sphere.}
We discretize the 2-sphere $S^2$ using an icosahedral geodesic mesh at subdivision level $n=4$, obtained by subdividing each face of the regular icosahedron into $n^2$ triangles and projecting the resulting vertices onto the unit sphere. The mesh consists of $N_0 = 162$ vertices, $N_1 = 480$ edges, and $N_2 = 320$ faces.

This triangulation is invariant under the full icosahedral group $I_h$ of order $|I_h|=120$, which contains the 60 proper rotations of $I$ and their 60 improper counterparts (compositions with the spatial inversion $i$). The 120 elements are enumerated from a standard generator set (5-fold, 3-fold, inversion).

The symmetry-adapted bases $Q_0$, $Q_1$, $Q_2$ for the cochain spaces $\Omega_d^0$, $\Omega_d^1$, $\Omega_d^2$ are constructed via the character-theoretic projection operators~\eqref{eq:projection}. Since $I_h = I \times \{E, i\}$, its ten irreducible representations are obtained by pairing each irrep of $I$ with the two one-dimensional representations of $\mathbb{Z}_2$: the five gerade (even) irreps $A_g$, $T_{1g}$, $T_{2g}$, $G_g$, $H_g$ and the five ungerade (odd) irreps $A_u$, $T_{1u}$, $T_{2u}$, $G_u$, $H_u$, all of dimensions $1, 3, 3, 4, 5$ respectively (with $2(1^2+3^2+3^2+4^2+5^2) = 120 = |I_h|$).

Rather than block-diagonalizing the composite operator $\Delta_1$ directly, we first verify that the DEC building blocks are individually block-diagonal in the symmetry-adapted basis. The relative off-block residuals of $d_0$ and $d_1$ are $6.9\times10^{-16}$ and $8.6\times10^{-16}$ respectively, confirming equivariance to machine precision (Theorem~\ref{thm:d_equiv}). The Hodge stars, being diagonal matrices on a symmetric mesh, satisfy the same property by construction ($\star_k$ commutes with $\rho^k(g)$ because symmetry-related simplices have equal metric volumes; residuals $\leq 9.4\times10^{-16}$ for all three).

Since $\Delta_1 = \delta_2 d_1 + d_0 \delta_1$ where $\delta_k = \star_{k-1}^{-1}\, d_{k-1}^\top\, \star_k$, the Hodge Laplacian is composed entirely from the block-diagonalized building blocks. We assemble $\Delta_1$ in the symmetry-adapted basis by composing the transformed operators:
\begin{equation*}
    [\Delta_1]_{\mathrm{sym}} = \tilde{\star}_1^{-1} \tilde{d}_1^\top \tilde{\star}_2\, \tilde{d}_1 \;+\; \tilde{d}_0\, \tilde{\star}_0^{-1} \tilde{d}_0^\top \tilde{\star}_1,
\end{equation*}
where $\tilde{d}_k = Q_{k+1}^\top d_k\, Q_k$ and $\tilde{\star}_k = Q_k^\top \star_k\, Q_k$. The resulting matrix is block-diagonal by construction, with off-block residual $1.7\times10^{-15}$, and agrees with the direct computation $Q_1^\top \Delta_1 Q_1$ to relative error $2.0\times10^{-15}$.

\paragraph{Torus with $D_{6h}$ symmetry.}
To demonstrate that the framework generalises beyond the sphere, we construct a torus of revolution $T^2$ embedded in $\mathbb{R}^3$ via the standard parametrization $\mathbf{r}(u,v) = ((R + r\cos v)\cos u,\; (R + r\cos v)\sin u,\; r\sin v)$ with major radius $R=3$ and minor radius $r=1$. This embedded torus carries a non-constant Gaussian curvature $K(v) = \cos v\,/\,[r(R + r\cos v)]$ that is positive on the outer rim ($v=0$), negative on the inner rim ($v=\pi$), and vanishes on the top and bottom parallels ($v=\pm\pi/2$)~\cite{docarmo2016}; it is therefore a curved test surface, distinct from the intrinsically flat $3$-torus of Section~\ref{sec:exp_3complex}. The triangulation uses a hexagonal (staggered) lattice with $36 \times 12 = 432$ vertices, $1296$ edges, and $864$ faces, yielding Euler characteristic $\chi = 432 - 1296 + 864 = 0$ as expected for genus~1. The mesh is invariant under the point group $D_{6h}$ of order $|D_{6h}|=24$, generated by a $60^\circ$ rotation $C_6$ about the torus axis, a $180^\circ$ rotation $C_2'$ about a radial axis, and the horizontal mirror $\sigma_h: z \mapsto -z$. The group has 12 conjugacy classes and 12 irreducible representations: four one-dimensional ($A_{1g}$, $A_{2g}$, $B_{1g}$, $B_{2g}$) and two two-dimensional ($E_{1g}$, $E_{2g}$) gerade irreps, plus their ungerade counterparts ($4 \cdot 1^2 + 2 \cdot 2^2 = 12$, and $2 \cdot 12 = 24 = |D_{6h}|$).

The choice of a hexagonal lattice, in which odd poloidal rows are staggered by half a grid step in the toroidal direction, is motivated by a numerical obstruction that arises with quad-based triangulations. When a quadrilateral cell is split into two triangles by a single diagonal, the two triangles sharing that diagonal form a parallelogram whose opposite angles sum to $\pi$. The cotangent formula for the discrete Hodge star on 1-forms, $\star_1(e) = (\cot\alpha_e + \cot\beta_e)/2$, then yields $\star_1(e) = 0$ on every diagonal edge---the cocircular, degenerate-Delaunay case, in which the opposite angles across the diagonal meet the local Delaunay bound $\alpha_e + \beta_e \leq \pi$ with equality, so the cotangent weight sits at the boundary of positivity \cite{bobenko2006,wardetzky2007}. Since the composition of $\Delta_1$ from building blocks requires the inversion of $\star_1$, vanishing diagonal entries render this composition numerically unstable. The hexagonal lattice avoids this issue entirely. All triangles are approximately equilateral, all cotangent weights are strictly positive, and $\star_1$ is everywhere invertible (see Section~\ref{sec:discussion} for further discussion).

On the torus, the same procedure is applied with $Q_0$, $Q_1$, $Q_2$ constructed from the $D_{6h}$ character table. The relative off-block residuals are $8.8\times10^{-16}$ for $d_0$ and $2.9\times10^{-15}$ for $\star_1$, confirming equivariance at machine precision on a surface with fundamentally different topology and curvature. The composed $\Delta_1$ has off-block residual $3.1\times10^{-15}$ and agrees with the direct computation to relative error $1.9\times10^{-15}$. The kernel of $\Delta_1$ on the torus is two-dimensional, reflecting $H^1(T^2) \cong \mathbb{R}^2$ (two independent non-contractible cycles), in contrast to the trivial $H^1(S^2) = 0$ on the sphere.

The predicted structure appears on both surfaces (Figure~\ref{fig:experiment1}). On the sphere, $Q_1^\top d_0\, Q_0$ takes the rectangular block-diagonal form that connects the isotypic components of $\Omega_d^0$ and $\Omega_d^1$, and the $\Delta_1$ composed from the building blocks exhibits ten square diagonal blocks. The torus follows the analogous construction, with 12 blocks corresponding to the $D_{6h}$ irreps. Classifying the eigenvalues by irrep (Figure~\ref{fig:experiment1_spectra}) reveals degeneracies $d_\mu \in \{1,3,3,4,5\}$ per parity sector on the sphere and $d_\mu \in \{1,1,1,1,2,2\}$ per parity sector on the torus. The consistency of the block-diagonal structure across two surfaces with different topology (genus 0 vs.~1), curvature (constant positive vs.~sign-changing Gaussian curvature), and symmetry group ($I_h$ vs.~$D_{6h}$) validates the framework.

\begin{figure}[H]
    \centering
    \includegraphics[width=\linewidth]{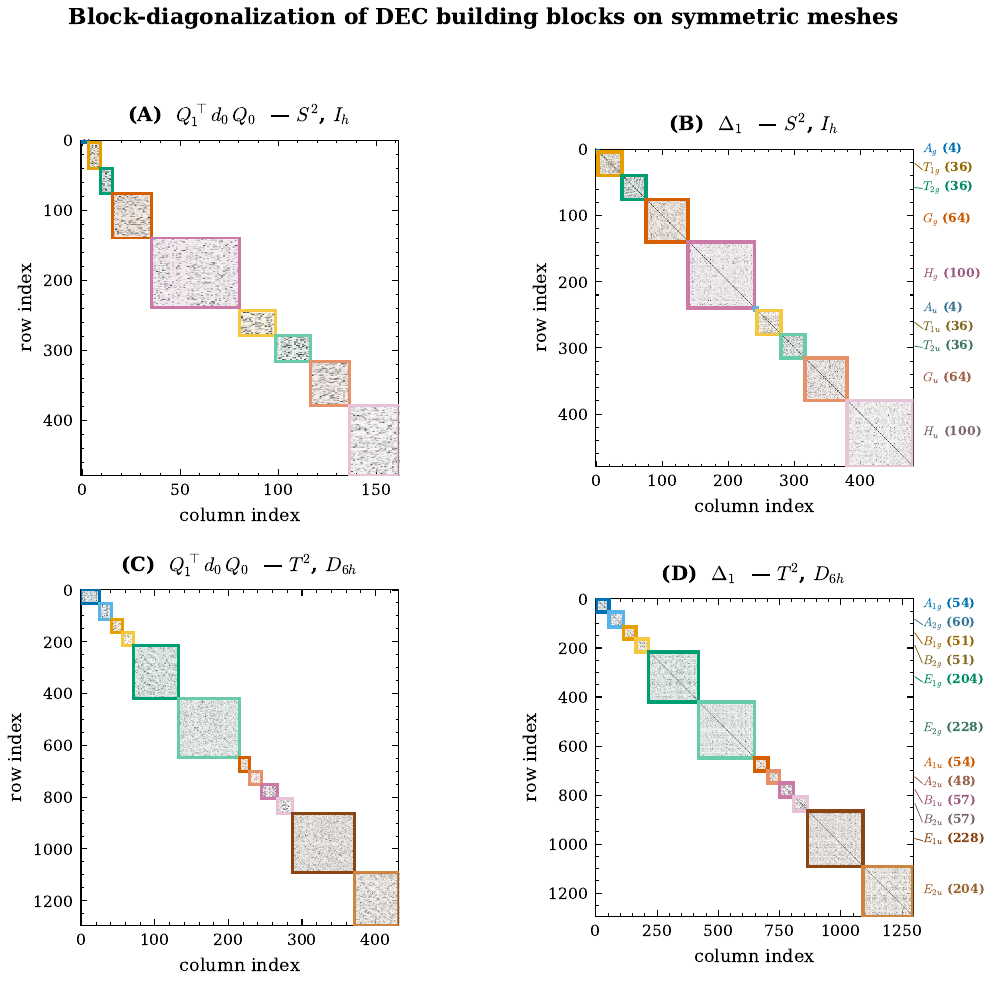}
    \caption{Block-diagonalization of the DEC building blocks on two symmetric surfaces. Panels~\textbf{(A)} and~\textbf{(B)} treat the icosahedral sphere $S^2$ at subdivision level $n=4$ ($N_1=480$ edges, $|I_h|=120$); panels~\textbf{(C)} and~\textbf{(D)} treat the torus $T^2$ on a $36\times12$ hexagonal lattice ($N_1=1296$, $|D_{6h}|=24$). \textbf{(A,\,C)}~The exterior derivative $Q_1^\top d_0\, Q_0$ in the symmetry-adapted basis, whose rectangular blocks connect the isotypic components of $\Omega_d^0$ and $\Omega_d^1$. \textbf{(B,\,D)}~The Hodge Laplacian $\Delta_1$ assembled from the transformed building blocks, with ten square diagonal blocks on the sphere and twelve on the torus. In every panel the horizontal axis is the column index and the vertical axis the row index, grey intensity encodes the magnitude of the matrix entry, and each block carries an outline and a light tint in a colour keyed to its irrep. The right-margin labels of~(B) name each block by irrep and size, reading $A_g$~(4), $T_{1g}$~(36), $T_{2g}$~(36), $G_g$~(64), $H_g$~(100) through the gerade sector and $A_u$~(4), $T_{1u}$~(36), $T_{2u}$~(36), $G_u$~(64), $H_u$~(100) through the ungerade sector; those of~(D) read $A_{1g}$~(54), $A_{2g}$~(60), $B_{1g}$~(51), $B_{2g}$~(51), $E_{1g}$~(204), $E_{2g}$~(228) and then $A_{1u}$~(54), $A_{2u}$~(48), $B_{1u}$~(57), $B_{2u}$~(57), $E_{1u}$~(228), $E_{2u}$~(204). A short leader line joins any label displaced to avoid overprinting back to the block it annotates. Entries outside the outlined blocks do not vanish identically but stay at relative magnitude $\sim 10^{-15}$, so the block-diagonalization is exact to machine precision.}
    \label{fig:experiment1}
\end{figure}

\begin{figure}[H]
    \centering
    \includegraphics[width=\linewidth]{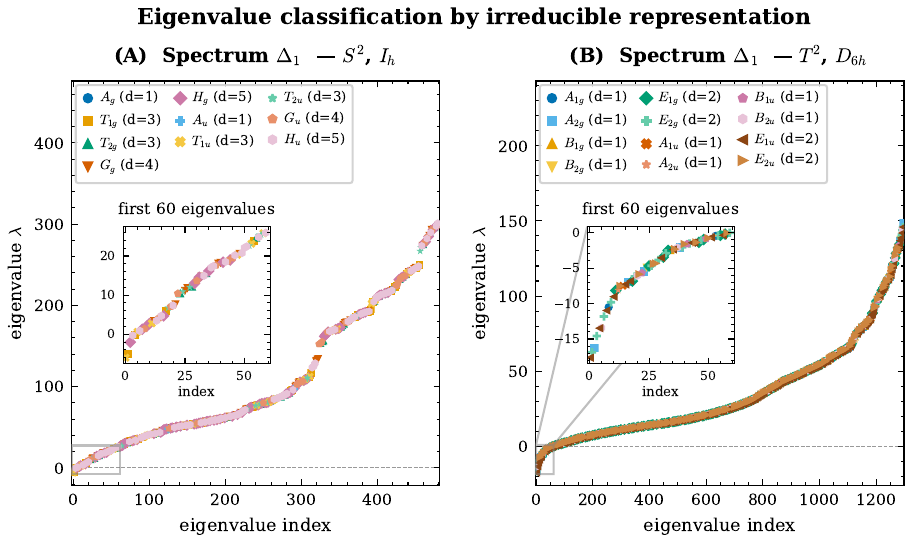}
    \caption{Classification of the $\Delta_1$ eigenvalues by irreducible representation, on the two meshes of Figure~\ref{fig:experiment1}. \textbf{(A)}~Icosahedral sphere $S^2$ under $I_h$. \textbf{(B)}~Torus $T^2$ under $D_{6h}$. Each panel plots the eigenvalue $\lambda$ against its index in ascending order. Every eigenvalue is drawn as a marker whose colour and shape both record the irrep of the block that produced it, so that the series stay separable in greyscale; the legend of each panel lists the irreps with their dimension, given as~$d$. The dashed horizontal line marks $\lambda=0$. An inset in each panel magnifies the first 60 eigenvalues, the grey rectangle in the main axes delimits the magnified window, and grey connector lines join that window to the inset. Every eigenvalue carries a definite irrep label, and the degeneracies are the ones Schur's lemma dictates, $d_\mu$-fold for the irrep~$\mu$, with $d_\mu \in \{1,3,3,4,5\}$ per parity sector on the sphere and $d_\mu \in \{1,1,1,1,2,2\}$ per parity sector on the torus.}
    \label{fig:experiment1_spectra}
\end{figure}

\subsection{Experiment 2: Eigenfunction classification by irreducible representation}

The block-diagonal structure established in Experiment~1 has a direct spectral consequence. Every eigenfunction of an equivariant operator belongs to exactly one irreducible representation of~$I_h$. Rather than computing the full spectrum and classifying eigenvectors \emph{a~posteriori}, the symmetry-adapted basis provides this classification for free---one simply diagonalises each reduced block independently.

We apply this procedure to the scalar Laplacian $\Delta_0 = \delta_1 d_0$ on the $n=4$ icosahedral mesh ($N_0 = 162$ vertices). The change of basis $Q_0^\top \Delta_0 Q_0$ yields ten diagonal blocks (one per irrep of $I_h$), and diagonalising each block produces eigenfunctions that are automatically labelled by their symmetry class. Figure~\ref{fig:experiment2} displays the lowest non-trivial eigenfunction from each irrep, rendered as a scalar field on the sphere.

The gerade eigenfunctions (top row) are invariant under spatial inversion $\mathbf{x} \mapsto -\mathbf{x}$, while the ungerade eigenfunctions (bottom row) change sign. Within each parity sector, the nodal structure reflects the dimension of the irrep: $A$-type eigenfunctions ($d_\mu=1$) have the fewest nodal lines, while $H$-type eigenfunctions ($d_\mu=5$) exhibit the most complex patterns. This classification is analogous to the Boson/Fermion decomposition of Schr\"odinger eigenfunctions on the cube demonstrated by Olver~\cite{olver2025}, but extended here to a curved surface with icosahedral symmetry and to the DEC framework.
The analogous classification holds on the torus under $D_{6h}$ (12 classes, $d_\mu \in \{1,1,1,1,2,2\}$ per parity sector); the visualisation is omitted as the qualitative conclusions are identical.

\begin{figure}[H]
    \centering
    \includegraphics[width=\linewidth]{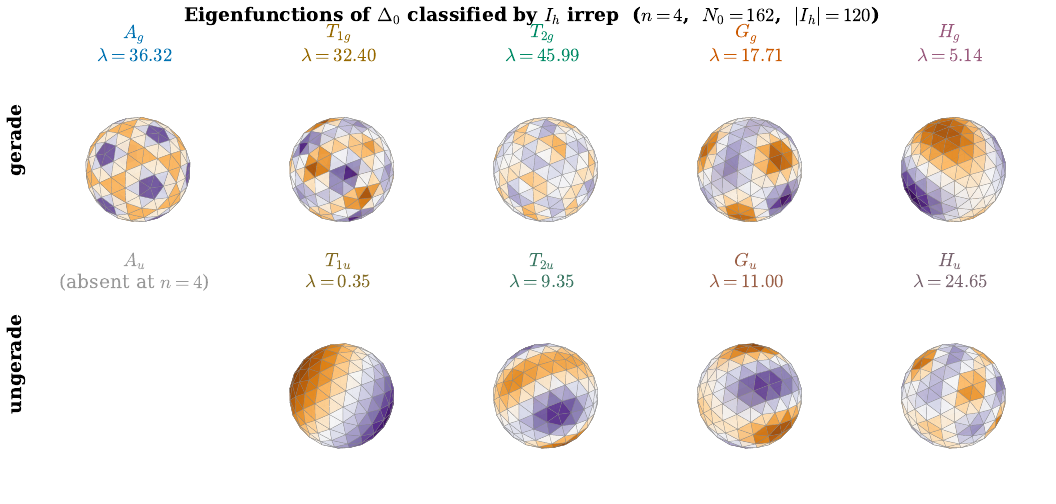}
    \caption{Lowest non-trivial eigenfunction of $\Delta_0$ in each irreducible representation of~$I_h$, rendered on the geodesic icosphere ($n=4$, $N_0=162$). Top row: gerade irreps (invariant under inversion). Bottom row: ungerade irreps (anti-invariant under inversion). The $A_u$ irrep has multiplicity zero at this resolution. Colour scale: purple (positive) to orange (negative).}
    \label{fig:experiment2}
\end{figure}

\paragraph{Spectral fingerprint of the subduction $D_\ell \!\downarrow\! I_h$.}
\label{sec:subduction}

The block classification of Experiment~2 tells us \emph{which} irrep an eigenfunction carries; representation theory predicts more. On the round sphere the eigenspaces of the Laplace--Beltrami operator are the spherical-harmonic degrees $D_\ell$, each an irreducible representation of $\mathrm{SO}(3)$ of dimension $2\ell+1$. Restricting the $\mathrm{SO}(3)$ action to the icosahedral subgroup decomposes $D_\ell$ into icosahedral irreps---the subduction $D_\ell \!\downarrow\! I_h$, equivalently the crystal-field splitting familiar from molecular physics~\cite{dresselhaus2008}. The smallest degree $\ell$ whose subduction first contains a given irrep $\mu$ is therefore the lowest spherical-harmonic degree at which $\mu$ can appear. We test the sharper claim that the \emph{lowest non-trivial} eigenfunction of $\Delta_0$ in each isotypic block concentrates its spectral weight at exactly that predicted degree $\ell_{\mathrm{pred}}$.

For each block we take the full degenerate multiplet at its lowest non-trivial eigenvalue and measure its per-degree content against a real spherical-harmonic design matrix $Y_\ell$ in the mass-weighted discrete inner product $\langle f,g\rangle_M = f^\top M g$, with $M = \star_0$ the dual-area mass matrix. The per-degree power is the trace of the $M$-orthogonal least-squares projection onto $\mathrm{span}\,Y_\ell$, $P(\ell) = \operatorname{tr}\!\big(b_\ell^\top G_\ell^{-1} b_\ell\big)\big/\operatorname{tr}\!\big(\Phi^\top M\,\Phi\big)$ with Gram matrix $G_\ell = Y_\ell^\top M\,Y_\ell$ and $b_\ell = Y_\ell^\top M\,\Phi$; the residual of the joint fit over all degrees serves as a quality gate. We weight by $M$ for two reasons. The icosphere vertices carry unequal dual areas, and the sampled real harmonics are not orthogonal, so the naive estimator $\sum_m|\langle\phi,Y_\ell^m\rangle|^2$ smears power across degrees. Treating the whole multiplet as a subspace (rather than a single eigenvector) keeps $P(\ell)$ basis-independent.

Table~\ref{tab:subduction} reports the measured dominant degree $\ell^\star$ at $N_0 = 362$ ($n=6$) against the subduction prediction. All ten irreps place their dominant degree exactly where $D_\ell \!\downarrow\! I_h$ predicts, the measured multiplet degeneracy equals the irrep dimension in every block, and the dominant-degree power fraction exceeds $0.98$ throughout. Under refinement from $N_0 = 162$ to $N_0 = 362$ the joint-fit residual tightens by an order of magnitude (from $\leq 0.156$ to $\leq 0.015$), confirming that the residual is a discretisation artifact rather than spectral leakage. Figure~\ref{fig:subduction} shows the full per-degree spectra. The pseudoscalar $A_u$ first subduces only at $\ell = 15$ (the lowest odd degree whose subduction contains the totally antisymmetric irrep~\cite{cohan1958}); its lowest-block multiplet is orthogonal to every degree $\ell < 15$ to within a joint-fit residual of $\sim 10^{-12}$, the spectral signature of a representation that no low harmonic can carry.

\begin{table}[h]
\centering
\small
\begin{tabular}{lrrrrr}
\toprule
Irrep & $\lambda$ & mult. & $\ell^\star$ & $P(\ell^\star)$ & $\ell_{\mathrm{pred}}$ \\
\midrule
$A_g$    & 37.788  & 1 & 6  & 0.987 & 6 \\
$T_{1g}$ & 37.971  & 3 & 6  & 0.997 & 6 \\
$T_{2g}$ & 60.124  & 3 & 8  & 0.997 & 8 \\
$G_g$    & 19.093  & 4 & 4  & 0.999 & 4 \\
$H_g$    & 5.809   & 5 & 2  & 0.997 & 2 \\
$A_u$    & 120.257 & 1 & 15 & 1.000 & 15 \\
$T_{1u}$ & 1.866   & 3 & 1  & 0.997 & 1 \\
$T_{2u}$ & 11.410  & 3 & 3  & 0.996 & 3 \\
$G_u$    & 11.701  & 4 & 3  & 0.999 & 3 \\
$H_u$    & 27.736  & 5 & 5  & 0.998 & 5 \\
\bottomrule
\end{tabular}
\caption{Measured spherical-harmonic content of the lowest non-trivial eigenfunction of $\Delta_0$ in each isotypic block on the geodesic icosphere ($n=6$, $N_0 = 362$). For every irrep the dominant degree $\ell^\star$ matches the subduction prediction $\ell_{\mathrm{pred}}$ (smallest $\ell>0$ whose $D_\ell \!\downarrow\! I_h$ contains the irrep), the multiplet degeneracy equals the irrep dimension, and the dominant-degree power fraction $P(\ell^\star)$ exceeds $0.98$. The pseudoscalar $A_u$, first admissible at $\ell=15$, is orthogonal to every lower degree to machine precision.}
\label{tab:subduction}
\end{table}

\begin{figure}[H]
    \centering
    \includegraphics[width=\linewidth]{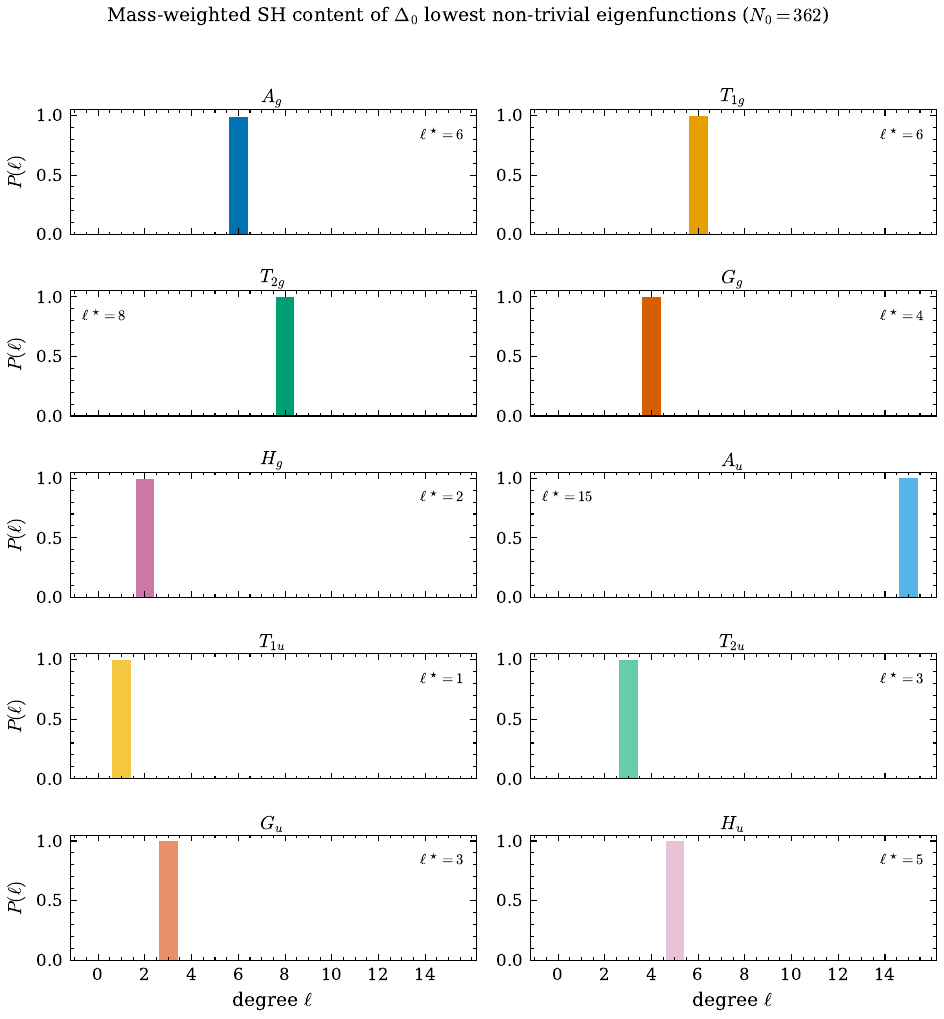}
    \caption{Mass-weighted per-degree power $P(\ell)$ of the lowest non-trivial eigenfunction of $\Delta_0$ in each icosahedral irrep, on the geodesic icosphere ($n=6$, $N_0=362$). Each panel is annotated with the measured dominant degree $\ell^\star$, which coincides with the subduction prediction in all ten cases. The concentration of power at a single degree turns the group-theoretic correlation diagram $D_\ell \!\downarrow\! I_h$ into a directly measurable spectral fingerprint.}
    \label{fig:subduction}
\end{figure}

A sampling limitation bounds the high-degree claims. At $N_0 = 362$ the degree-$\ell$ Gram matrix $G_\ell$ becomes ill-conditioned as the harmonic count $K = (\ell+1)^2$ approaches $N_0$, so $\ell = 15$ (where $K = 256$) is the sampling ceiling. The unit power fraction reported for $A_u$ at $\ell=15$ therefore does not separate $\ell=15$ from $\ell>15$; the well-sampled statement is orthogonality to every degree $\ell < 15$, which is exactly what the subduction predicts for the pseudoscalar.

\subsection{Experiment 3: Computational performance}

The block-diagonal structure produced by the symmetry-adapted basis directly translates into computational savings. Instead of factorising one $N_0 \times N_0$ system, we factorise several independent blocks whose dimensions are at most $\sim N_0 d_{\max}^2 / |G|$, where $d_{\max}$ is the largest irrep dimension. Because dense Cholesky factorisation costs $\tfrac{1}{3}n^3$ floating-point operations (FLOPs), the cubic scaling makes the decomposition increasingly advantageous as~$N_0$ grows.

We solve the regularised Poisson equation $(\Delta_0 + \varepsilon I)\phi = \rho$, with regularisation $\varepsilon = 10^{-6}$ (the same value used in Section~\ref{sec:exp_3complex}) lifting the constant null mode that $\Delta_0$ carries on a closed manifold, and with right-hand side $\rho$ a generic random cochain drawn from the standard normal distribution; the FLOP counts and the solve cost are independent of the particular $\rho$. We compare three strategies:
\begin{itemize}
    \item[(a)] \emph{Direct}: dense symmetric Cholesky solve of the full $N_0 \times N_0$ system.
    \item[(b)] \emph{Sequential blocks}: factorise each diagonal block independently and sum the costs.
    \item[(c)] \emph{Parallel blocks}: the cost equals that of the largest block (ideal parallelisation).
\end{itemize}
To ensure reproducibility, we report the theoretical FLOP count $F = \tfrac{1}{3}n^3 + 2n^2$ (Cholesky factorisation plus two triangular solves) rather than wall-clock time.\footnote{FLOP counts depend only on the block dimensions dictated by the character table, and are therefore deterministic and independent of hardware, compiler, BLAS implementation, and concurrency runtime; wall-clock measurements would introduce variability unrelated to the algebraic reduction studied here.}

All three strategies use dense factorization, of the full operator in strategy~(a) and of each isotypic block in strategies~(b) and~(c), so the speedups below are FLOP reductions relative to a dense factorization. This is the like-for-like comparator, because the symmetry-adapted change of basis $Q$ mixes the sparse $\Delta_0$ into dense isotypic blocks. A practitioner would instead use a sparse direct or iterative solver, for instance multigrid, that exploits the $O(N_0)$ nonzeros of $\Delta_0$; we do not benchmark that path here. As noted in Section~\ref{subsec:complexity}, the block-diagonalization is orthogonal to sparsity and composes with such a solver, so the figures below isolate the symmetry gain in the dense regime rather than claiming an absolute advantage over the fastest available solver.

\paragraph{Icosahedral sphere ($I_h$).}
Table~\ref{tab:performance_sphere} reports the FLOP counts for five subdivision levels ($n = 2, \ldots, 6$; $N_0 = 42, \ldots, 362$). The parallel FLOP speedup grows from $18\times$ at $n=2$ to $62\times$ at $n=6$, reflecting the fact that the largest block (the $H_g$ isotypic component, of leading dimension~$25N_0/120$) grows much more slowly than the full problem size. The number of active blocks reaches 10 at $n=6$ when the $A_u$ irrep first acquires nonzero multiplicity.

\begin{table}[h]
\centering
\small
\begin{tabular}{rrrrrr}
\toprule
$n$ & $N_0$ & blocks & max block & $S_{\mathrm{seq}}$ (FLOPs) & $S_{\parallel}$ (FLOPs) \\
\midrule
2 &  42 &  7 &  15 & 13.6 & 17.9 \\
3 &  92 &  9 &  25 & 26.3 & 42.8 \\
4 & 162 &  9 &  45 & 27.6 & 42.7 \\
5 & 252 &  9 &  65 & 32.4 & 54.6 \\
6 & 362 & 10 &  90 & 34.7 & 62.0 \\
\bottomrule
\end{tabular}
\caption{FLOP-based speedups for $(\Delta_0 + \varepsilon I)\phi = \rho$ on icosahedral meshes ($I_h$, 10 irreps). $S_{\mathrm{seq}}$ = direct FLOPs / sum of block FLOPs; $S_{\parallel}$ = direct FLOPs / max block FLOPs.}
\label{tab:performance_sphere}
\end{table}

\paragraph{Torus ($D_{6h}$).}
We repeat the experiment on the hexagonal-lattice torus with five grid resolutions ($18 \times 6$ through $108 \times 36$; $N_0 = 108, \ldots, 3888$). Although $|D_{6h}| = 24$ is much smaller than $|I_h| = 120$, the group has 12 irreps (all of dimension $\leq 2$), which distributes the cochain space into more numerous and smaller blocks. As shown in Table~\ref{tab:performance_torus}, the parallel FLOP speedup reaches $182\times$ at the largest resolution, higher than on the sphere at comparable sizes. This gain owes to the more uniform block distribution. The largest block (the $E_{2g}$ isotypic component) accounts for only $\sim 18\%$ of the total degrees of freedom, compared to $\sim 25\%$ for the $H_g$ block on the sphere.

\begin{table}[h]
\centering
\small
\begin{tabular}{rrrrrr}
\toprule
Grid & $N_0$ & blocks & max block & $S_{\mathrm{seq}}$ (FLOPs) & $S_{\parallel}$ (FLOPs) \\
\midrule
$18\times 6$   &   108 & 12 &   24 &  31.6 &   77.0 \\
$36\times 12$  &   432 & 12 &   84 &  45.1 &  128.7 \\
$54\times 18$  &   972 & 12 &  180 &  48.8 &  153.3 \\
$72\times 24$  &  1728 & 12 &  312 &  50.3 &  167.3 \\
$108\times 36$ &  3888 & 12 &  684 &  51.4 &  182.3 \\
\bottomrule
\end{tabular}
\caption{FLOP-based speedups for $(\Delta_0 + \varepsilon I)\phi = \rho$ on hexagonal-lattice torus meshes ($D_{6h}$, 12 irreps). The higher speedup compared to the sphere reflects the more uniform distribution of degrees of freedom across 12 irreps, all of dimension $\leq 2$.}
\label{tab:performance_torus}
\end{table}

Figure~\ref{fig:experiment3} presents the FLOP counts and speedup factors for both surfaces. The direct cost grows as $\Theta(N_0^3)$, while the block-wise cost grows much more slowly because the block dimensions scale as $O(N_0/|G|)$ in the leading term. Since each isotypic block has asymptotic dimension $\dim V_\mu^k \sim d_\mu^2 N_0 / |G|$, the sequential speedup converges to $|G|^3/\!\sum_\mu d_\mu^6$ and the parallel speedup to $|G|^3/d_{\max}^6$ as $N_0 \to \infty$, where $d_{\max} = \max_\mu d_\mu$. For $I_h$ ($d_{\max}=5$) these asymptotic limits are $\approx 41$ (sequential) and $\approx 111$ (parallel); for $D_{6h}$ ($d_{\max}=2$) they are $\approx 52$ and $\approx 216$, respectively. The observed speedups at the largest resolutions ($62\times$ for the sphere, $182\times$ for the torus) are approaching these bounds from below, as expected for moderate~$N_0$. The cost of assembling the change-of-basis matrix~$Q_0$ is excluded, as it is amortised over all subsequent operator solves on the same mesh. Reducing each isotypic block to its multiplicity space, as described in Section~\ref{subsec:complexity}, raises the sequential asymptotes further, to $|G|^3/\sum_\rho d_\rho^3 \approx 3541$ for $I_h$ and $\approx 346$ for $D_{6h}$.

\begin{figure}[H]
    \centering
    \includegraphics[width=\linewidth]{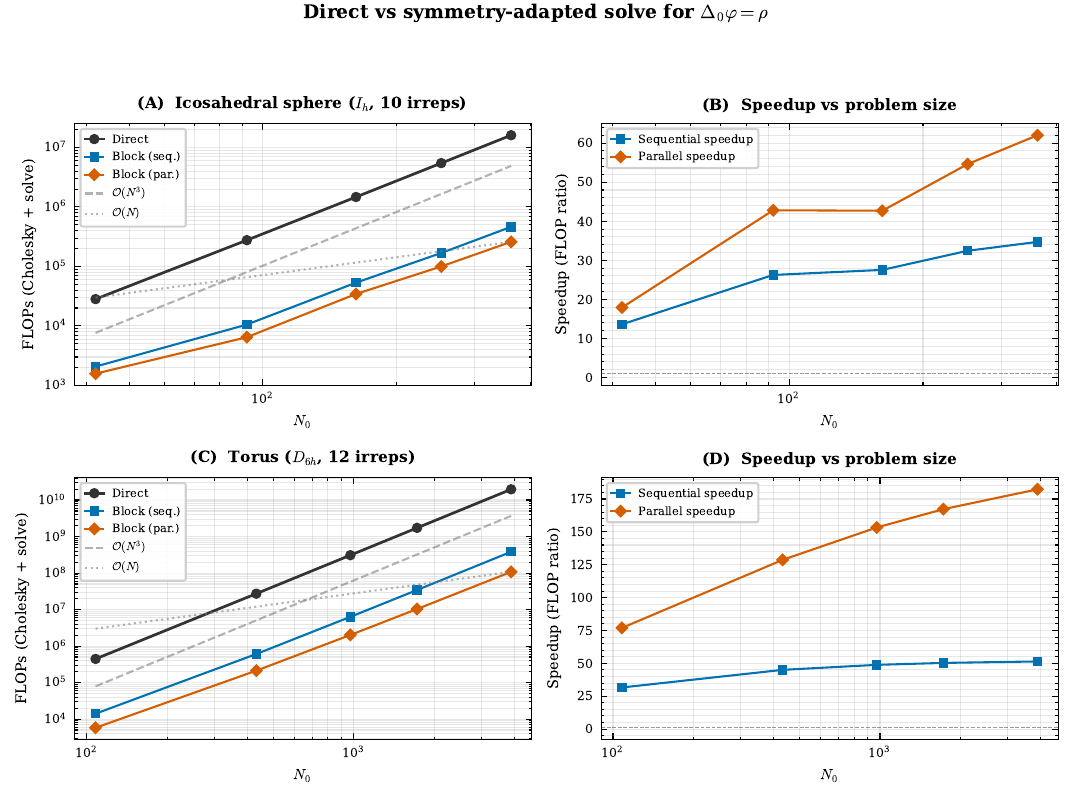}
    \caption{Computational performance of direct vs.\ symmetry-adapted solve for $\Delta_0 \phi = \rho$ on two surfaces. \textbf{Top row:} icosahedral sphere ($I_h$, 10 irreps). \textbf{Bottom row:} torus ($D_{6h}$, 12 irreps). \textbf{(A,\,C)}~FLOP count (log-log) as a function of~$N_0$. \textbf{(B,\,D)}~FLOP-based speedup: sequential (sum of block costs) and parallel (largest block cost). The torus achieves higher speedup than the sphere due to its more numerous irreps of smaller dimension.}
    \label{fig:experiment3}
\end{figure}

To confirm that the FLOP reduction translates into wall-clock gains, we timed the dense Cholesky solve of $(K+\varepsilon I)\phi=\rho$, with $K=\star_0\Delta_0=d_0^\top\star_1 d_0$ the symmetric positive-definite form of the Poisson operator, against its isotypic blocks, using single-threaded BLAS and the median of auto-batched repetitions. The measured parallel speedup reaches about $100\times$ on the torus at the largest resolutions, $13.6\times$ on the sphere at $N_0=362$, and $5.4\times$ on the fan-BCC tessellation of $T^3$ (Section~\ref{sec:exp_3complex}) at $N_0=432$ (Figure~\ref{fig:wallclock}). It stays below the FLOP prediction at every resolution. On the torus it grows and closes on the prediction through $N_0=1728$, where it reaches about $101\times$ against a predicted $167\times$; at the largest resolution $N_0=3888$ it plateaus near $100\times$ while the prediction rises to $182\times$, so the achieved fraction of the FLOP bound drops at this last point rather than continuing to narrow. The FLOP counts of Tables~\ref{tab:performance_sphere}--\ref{tab:performance_torus} therefore remain a reproducible upper bound, approached most closely at intermediate torus sizes. On the smallest meshes the direct solve is already sub-millisecond, where fixed overheads dominate and the sequential-block strategy falls below unit speedup. Timings were obtained on an AMD Ryzen~7 3700U (eight logical cores, 9.7~GB RAM) with OpenBLAS~0.3.31 restricted to one thread, under NumPy~2.4.4 and SciPy~1.17.1.

\begin{figure}[H]
    \centering
    \includegraphics[width=3.5in]{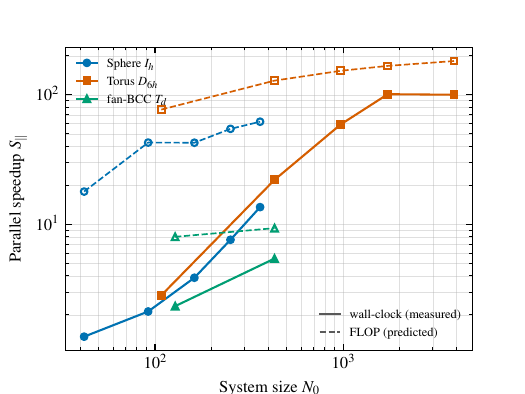}
    \caption{Measured wall-clock parallel speedup versus the predicted FLOP speedup for the block-diagonal solve of $(K+\varepsilon I)\phi=\rho$, as a function of the number of $0$-cochain degrees of freedom $N_0$ (log-log). Solid curves with filled markers are the measured wall-clock speedup $S_\parallel=t_{\mathrm{direct}}/t_{\text{max-block}}$; dashed curves with hollow markers are the FLOP-count prediction. Colour and marker encode the case, with sphere $I_h$ shown as blue circles, torus $D_{6h}$ as vermillion squares, and fan-BCC $T_d$ as green triangles. Solves used single-threaded OpenBLAS. The measured speedup lies below the FLOP prediction at every size; on the torus it grows and narrows the gap through $N_0=1728$, then plateaus near $100\times$ at $N_0=3888$ as the prediction continues to rise.}
    \label{fig:wallclock}
\end{figure}

\subsection{Experiment 4: 3-complex extension under $T_d$ symmetry}
\label{sec:exp_3complex}
We extend the numerical verification to a three-dimensional simplicial complex, confirming Theorems~\ref{thm:d_equiv}, \ref{thm:star_equiv}, and~\ref{thm:universal} in this higher-dimensional setting.

\paragraph{Test problem.}
The domain is the flat 3-torus $T^3 = \mathbb{R}^3/\mathbb{Z}^3$ tessellated by the fan-BCC triangulation. Each unit cube is divided into 12 tetrahedra by the body-centred-cubic rule, producing a purely simplicial 3-complex. The measured Hodge-Laplacian kernel dimensions $\dim\ker\Delta_k = 1,3,3$ for $k=0,1,2$ match the Betti numbers of the $3$-torus at every resolution, and Poincar\'e duality fixes $\beta_3=\beta_0=1$, so the mesh reproduces the full signature $(\beta_0,\beta_1,\beta_2,\beta_3)=(1,3,3,1)$ of $T^3$. The symmetry group is the tetrahedral point group $T_d$ (order~24, five irreps $A_1,A_2,E,T_1,T_2$ with dimensions $1,1,2,3,3$). The BCC lattice itself carries the full octahedral point group $O_h$ of order~48. The diagonal-splitting rule is invariant only under its tetrahedral subgroup $T_d$, because a four-fold rotation of a cube face reverses the chosen diagonal. Because the fan-BCC tessellation is not well-centered, the circumcentric dual of Definition~\ref{def:hodge} can develop negative dual volumes~\cite{vanderzee2010}, so the Hodge stars here use the standard barycentric (centroid) dual, in which $\star\sigma$ is spanned by the barycenters of the cofaces of $\sigma$. The barycenters of any chain of cofaces are affinely independent, so the barycentric dual cells are non-degenerate and the ratios $|\star\sigma|/|\sigma|$ of~\eqref{eq:hodge} are strictly positive for every $\sigma$, independently of well-centeredness. Each $\star_k$ is therefore positive-diagonal and invertible, and $\delta_k,\Delta_k$ are well-defined on the mesh. Positivity is confirmed numerically at every resolution ($\min_\sigma(\star_k)_{\sigma\sigma}>0$ for $k=0,1,2,3$). By Corollary~\ref{cor:dual_agnostic}, the equivariance of Theorem~\ref{thm:star_equiv} extends to this dual, since an isometry sends the barycenter of a coface to that of its image, $g(b_\tau)=b_{g\tau}$, exactly as for circumcenters, and preserves the barycentric cell volumes.

The flat 3-torus with a body-centred-cubic tessellation is the standard periodic cell of a cubic crystal, and the model problem $(\Delta_0+\varepsilon I)\phi=\rho$ can be read as the discrete screened Poisson (Debye--H\"uckel) equation for the electrostatic potential $\phi$ of a periodic charge density $\rho$, with $\varepsilon$ in the role of the inverse squared screening length. This reading fixes the physical interpretation of the model problem; it is not a claim that $\Delta_k$ on the barycentric fan-BCC mesh is a validated electrostatics solver. The barycentric dual trades the primal--dual orthogonality of the circumcentric construction (Remark~\ref{rmk:wellcentered}) for guaranteed positivity, so each $\star_k$ is positive and invertible but is not shown to be a consistent discretization; whether $\Delta_k$ approximates the continuum screened Poisson operator on this non-well-centered mesh is the open convergence question of Section~\ref{sec:discussion}. This section accordingly establishes structure (equivariance, block decomposition, and cost), not approximation accuracy. The block-diagonalization applies directly to periodic electrostatics in a crystal of $T_d$ site symmetry, a setting where symmetry-adapted bases are already standard practice. The value $\varepsilon=10^{-6}$ used here corresponds to weak screening and regularises the constant null mode that $\Delta_0$ carries on a closed manifold. The right-hand side is generic, since the cost reduction depends on the operator and the group action rather than on a particular charge configuration.

The fan-BCC construction is valid only for even $n \geq 4$; three resolutions $n \in \{4,6,8\}$ are used, corresponding to $N_0 = 128, 432, 1024$ vertices. At $n=8$ ($N_0 = 1024$) only the equivariance verification was performed; the dense $12288 \times 12288$ character projector for the $\Delta_2$ block decomposition, together with its eigensolver workspace, exceeds the $9.7$~GB of memory available on the workstation used here, so the full block-decomposition pipeline becomes memory-bound at this resolution. The full pipeline was run at $n=4$ and $n=6$, where $n=6$ already attains the regular-representation asymptote (Figure~\ref{fig:experiment4_speedup}).

\paragraph{Equivariance.}
Theorems~\ref{thm:d_equiv} and~\ref{thm:star_equiv} predict that the equivariance defects $[\rho^{k+1}(g),d_k]$ and $[\rho^{n-k}(g),\star_k]$ vanish exactly for every $g \in T_d$. Table~\ref{tab:equivariance_3d} reports the maximum residual over all 24 group elements at three mesh resolutions, all at machine precision. The composite operator $\Delta_k$ inherits equivariance via Lemma~\ref{lem:delta_equiv}; the slight precision loss for $\Delta_k$ relative to $d_k$ and $\star_k$ reflects accumulated round-off in the matrix products that define it.
\begin{table}[h]
\centering
\small
\begin{tabular}{ccrrr}
\toprule
$N_0$ & $k$ & {$\max_g \|[\rho^{k+1}(g),d_k]\|_\infty$} & {$\max_g \|[\rho^{n-k}(g),\star_k]\|_\infty$} & {$\max_g \|[\rho^k(g),\Delta_k]\|_\infty$} \\
\midrule
128 & 0 & 0.00e+00 & 0.00e+00 & 7.11e-15 \\
128 & 1 & 0.00e+00 & 3.89e-16 & 4.26e-14 \\
128 & 2 & 0.00e+00 & 0.00e+00 & 1.42e-14 \\
432 & 0 & 0.00e+00 & 0.00e+00 & 7.11e-15 \\
432 & 1 & 0.00e+00 & 3.89e-16 & 4.97e-14 \\
432 & 2 & 0.00e+00 & 0.00e+00 & 2.13e-14 \\
1024 & 0 & 0.00e+00 & 0.00e+00 & 1.07e-14 \\
1024 & 1 & 0.00e+00 & 6.66e-16 & 8.53e-14 \\
1024 & 2 & 0.00e+00 & 0.00e+00 & 2.84e-14 \\
\bottomrule
\end{tabular}
\caption{Equivariance residuals for $d_k$, $\star_k$, and $\Delta_k$ on the fan-BCC tessellation of $T^3$ under $T_d$ symmetry at three mesh resolutions ($N_0=128,432,1024$). All commutator norms are at machine precision ($\leq 8.53 \times 10^{-14}$), confirming equivariance across all three Laplacian degrees and all 24 group elements.}
\label{tab:equivariance_3d}
\end{table}

\paragraph{Block decomposition and speedup.}
The character projectors split each Hodge-Laplacian into the five isotypic blocks of $T_d$, whose sizes reproduce the multiplicities predicted by the character table, with the residual outside the blocks at machine precision (Figure~\ref{fig:experiment4_blocks}). Solving each block separately reduces the arithmetic cost of the corresponding linear solve, and the resulting FLOP speedup $\rho_{\mathrm{FLOP}}$ grows with mesh resolution toward the regular-representation asymptote (Table~\ref{tab:speedup_3d}, Figure~\ref{fig:experiment4_speedup}). The wall-clock realisation of the same block-diagonal solve on this mesh is reported alongside the two surface cases in Figure~\ref{fig:wallclock}.

\begin{figure}[h]
    \centering
    \includegraphics[width=0.9\linewidth]{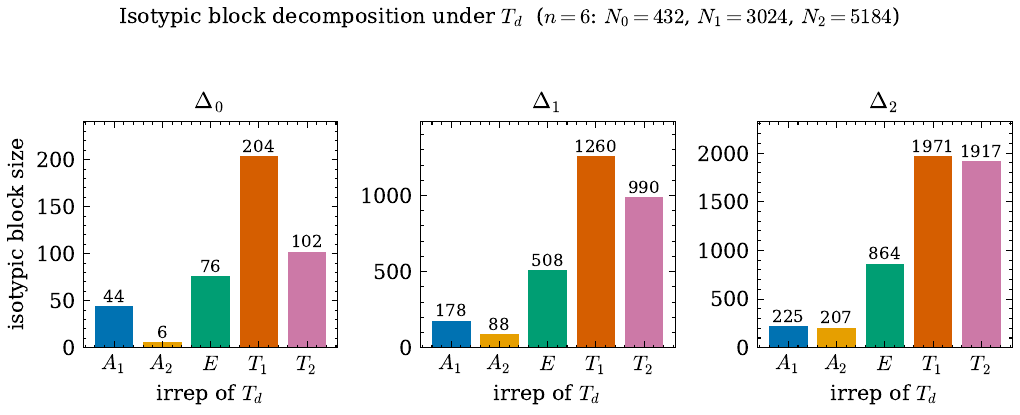}
    \caption{Block sizes of the isotypic decomposition of $\Delta_0$, $\Delta_1$, $\Delta_2$ under $T_d$ on the BCC mesh at resolution $n=6$ ($N_0 = 432$, $N_1 = 3024$, $N_2 = 5184$). For each Laplacian, the cochain space decomposes into five blocks indexed by the irreducible representations $A_1$, $A_2$, $E$, $T_1$, $T_2$ of $T_d$ with respective dimensions 1, 1, 2, 3, 3. Block sizes equal the predicted multiplicities from the character table to machine precision.}
    \label{fig:experiment4_blocks}
\end{figure}

\begin{table}[h]
\centering
\small
\begin{tabular}{ccccccc}
\toprule
$N_0$ & $N_k$ & $k$ & blocks $(A_1,A_2,E,T_1,T_2)$ & {$\rho_{\mathrm{FLOP}}$} & $\rho^{\mathrm{Schur}}_{\mathrm{seq}}$ & $\rho^{\mathrm{Schur}}_{\mathrm{par}}$ \\
\midrule
128 & 128 & 0 & $(19,\,1,\,24,\,63,\,21)$ & 7.5 & 115.3 & 226.4 \\
128 & 896 & 1 & $(62,\,22,\,152,\,390,\,270)$ & 8.7 & 199.0 & 327.4 \\
128 & 1536 & 2 & $(68,\,60,\,256,\,588,\,564)$ & 9.1 & 215.7 & 481.3 \\
432 & 432 & 0 & $(44,\,6,\,76,\,204,\,102)$ & 8.0 & 163.2 & 256.4 \\
432 & 3024 & 1 & $(178,\,88,\,508,\,1260,\,990)$ & 8.9 & 208.3 & 373.2 \\
432 & 5184 & 2 & $(225,\,207,\,864,\,1971,\,1917)$ & 9.1 & 215.9 & 491.2 \\
\bottomrule
\end{tabular}
\caption{FLOP-based speedup ratios for block-diagonalized solve of $(\Delta_k + \varepsilon I)\phi = \rho$ on the fan-BCC $T^3$ mesh under $T_d$ at resolutions $n=4$ ($N_0=128$) and $n=6$ ($N_0=432$). The isotypic block dimensions match character-table predictions; speedup approaches the sequential asymptote $\rho_\infty \approx 9.07$ derived in Figure~\ref{fig:experiment4_speedup}. The two rightmost columns give the Schur-multiplicity refinement (Section~\ref{subsec:complexity}): $\rho^{\mathrm{Schur}}_{\mathrm{seq}}$ factorizes one $c_\mu \times c_\mu$ block per irrep in place of the full isotypic block, with sequential limit $|G|^3/\sum_\rho d_\rho^3 = 216$; $\rho^{\mathrm{Schur}}_{\mathrm{par}}$ is the corresponding parallel ratio, limit $|G|^3/d_{\max}^3 = 512$.}
\label{tab:speedup_3d}
\end{table}

\begin{figure}[h]
    \centering
    \includegraphics[width=0.6\linewidth]{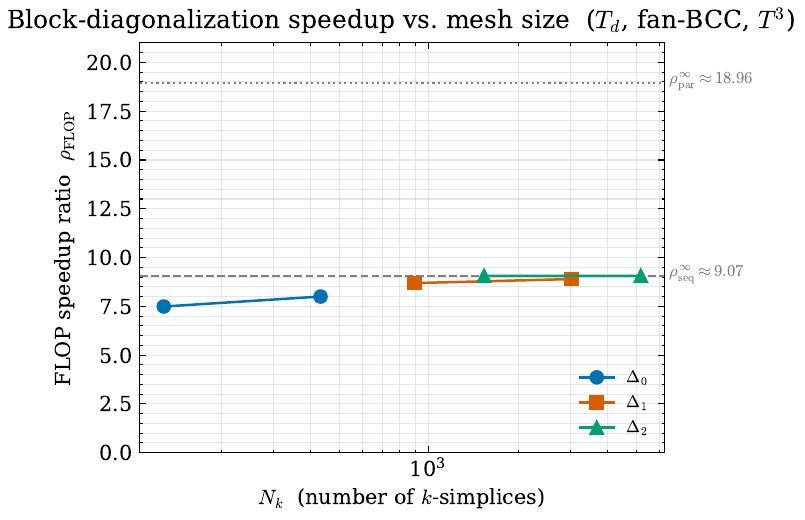}
    \caption{FLOP speedup ratio $\rho_{\text{FLOP}}$ for the block-diagonalized sequential solve of $\Delta_0$, $\Delta_1$, $\Delta_2$ on the fan-BCC tessellation of $T^3$ under $T_d$. The dashed line marks the regular-representation sequential limit $\rho^\infty_{\text{seq}} = |G|^3/\sum_\rho d_\rho^6 \approx 9.07$; the upper bound $\rho^\infty_{\text{par}} = |G|^3/d_{\max}^6 \approx 18.96$ is attainable by solving the isotypic blocks in parallel. Convergence exponents $\alpha$ (fitted via $\rho_\infty - \rho \propto N_k^{-\alpha}$) are annotated in the legend.}
    \label{fig:experiment4_speedup}
\end{figure}

Two features of the speedup data merit interpretation. First, $\Delta_2$ saturates the regular-representation asymptote already at $n=6$ while $\Delta_0$ and $\Delta_1$ approach it more slowly. The asymptote $\rho_\infty = |G|^3 / \sum_\rho d_\rho^6$ is attained when the isotypic block sizes are proportional to $d_\rho^2$; on the BCC mesh, the block multiplicities of $\Delta_2$ are nearly uniform across the higher-dimensional irreps and already approximate this regime, whereas $\Delta_0$ and $\Delta_1$ retain multiplicities skewed toward the one-dimensional irreps $A_1, A_2$, slowing the approach. Second, the gap between the sequential and parallel limits ($\rho_\infty^{\mathrm{seq}} \approx 9.07$ versus $\rho_\infty^{\mathrm{par}} \approx 18.96$) quantifies a further factor-of-two acceleration obtainable by solving the five isotypic blocks concurrently, requiring no algorithmic complexity beyond a five-way task parallelism.

Carried to its natural conclusion, the same construction admits a second reduction inside each isotypic component. By Schur's lemma an equivariant operator restricted to the component of irrep $\mu$ acts as $M_\mu \otimes I_{d_\mu}$, that is, as $d_\mu$ identical copies of a single $c_\mu \times c_\mu$ matrix $M_\mu$ on the multiplicity space, with multiplicity $c_\mu$ the isotypic block size divided by $d_\mu$. Factoring one copy per irrep rather than the whole isotypic block replaces the cost $\sum_\mu (c_\mu d_\mu)^3$ by $\sum_\mu c_\mu^3$. The sequential asymptote rises from the isotypic $|G|^3/\sum_\mu d_\mu^6 \approx 9.07$ to $|G|^3/\sum_\mu d_\mu^3 = 216$, a further factor of about $|G|$, and the parallel limit to $|G|^3/d_{\max}^3 = 512$. As with the isotypic curve, $\Delta_2$ saturates this refined asymptote already at $n=6$, reaching $215.9$, and Table~\ref{tab:speedup_3d} and Figure~\ref{fig:experiment4_speedup} now carry both tiers. Built explicitly from the partner (transfer) projector, $M_\mu$ recovers every eigenvalue of the full operator with $d_\mu$-fold multiplicity to machine precision (worst case $1.3\times 10^{-12}$), so $d_\mu-1$ of every $d_\mu$ blocks are redundant rather than merely counted away. In a linear solve the $d_\mu$ copies share one factorization and need $d_\mu$ back-substitutions, an $\mathcal{O}(d_\mu c_\mu^2)$ sub-leading term that leaves the leading order and asymptote unchanged.

\section{Discussion}
\label{sec:discussion}

The experiments confirm that the equivariance proved in Theorems~\ref{thm:d_equiv}--\ref{thm:star_equiv} propagates automatically to every operator assembled from $d$ and $\star$, including the Hodge Laplacians, and that the resulting block-diagonal systems approach the asymptotic bounds derived in Section~\ref{subsec:complexity}. The framework occupies a complementary position to existing symmetry-exploitation techniques in numerical PDEs, as the following comparisons illustrate.

DEC sits closest to spectral block-diagonalization methods. Olver's spectral approach \cite{olver2025} achieves block-diagonalization on flat Platonic domains (squares, cubes) using global polynomial bases, with the advantage of spectral convergence. Our DEC approach trades spectral convergence for geometric generality, applying to any simplicial mesh that faithfully represents the symmetry group, including meshes on curved manifolds where global polynomial bases are unavailable. A direct numerical comparison between the two frameworks is not well-defined because the underlying function spaces differ. Olver's construction operates on $L^2$ of polynomial spaces over flat polytopes, while ours operates on cochain spaces $\Omega_d^k(K)$ over simplicial complexes on curved manifolds. The two approaches address disjoint portions of the space of PDE problems, and their computational gains are not commensurable.

Finite Element Exterior Calculus offers a second point of comparison. The FEEC framework \cite{arnold2006,arnold2010} provides a Hilbert-complex approach to discretizing the de~Rham complex. While FEEC is more general than DEC in terms of approximation order, extending equivariance to higher-order FEEC discretizations requires explicitly classifying and constructing symmetry-adapted bases for complex spaces of polynomial differential forms. As recently detailed by Berchenko-Kogan \cite{berchenko2024} and Licht \cite{licht2024}, this task is algebraically intensive and highly dependent on the specific polynomial degree and element type. In contrast, DEC has the advantage of geometric simplicity, since the operators $d$ and $\star$ are defined directly on topological cochains without weak formulations, projection operators, or spatial polynomial algebras. This simplicity makes the equivariance analysis universal; a single, purely combinatorial symmetry-adapted basis computation completely block-diagonalizes the system, avoiding the degree-dependent algebraic complexity inherent to higher-order FEEC.

The construction also constrains mesh design through the invertibility of the Hodge star. The compositional pathway $\delta_k = \star_{k-1}^{-1} d_{k-1}^\top \star_k$ requires $\star_1$ to be everywhere invertible. The cotangent weights $\star_1(e) = (\cot\alpha_e + \cot\beta_e)/2$ vanish when the two angles opposite an interior edge sum to exactly $\pi$ \cite{bobenko2006,wardetzky2007}, rendering $\star_1$ singular and the operator composition unstable (see Section~\ref{sec:experiments} for the quad-based torus case). This positivity condition ties mesh design to the intrinsic Delaunay property \cite[Prop.~17]{bobenko2006}. The local Delaunay criterion \cite[Lem.~9]{bobenko2006} requires that the sum of opposite angles across any interior edge not exceed $\pi$, a condition weaker than global acute-angle constraints that permits obtuse triangles. Positivity of $\star_1$, however, needs the \emph{strict} inequality $\alpha_e + \beta_e < \pi$. The degenerate, cocircular case of equality is Delaunay yet yields $\star_1(e) = 0$, so invertibility of $\star_1$ is guaranteed by the strictly Delaunay property, and this degenerate case must be excluded explicitly on structured grids designed for symmetry preservation.

This study concerns the cost and structure of the symmetry reduction rather than the order of accuracy of the discretization. The symmetry-adapted solve is algebraically exact relative to the direct solve, so the solution inherits the accuracy of the underlying DEC discretization, whose convergence for Hodge--Laplace problems is established in~\cite{mohamed2018,schulz2018}.

Symmetry has a second consequence beyond cost. On well-centered meshes over contractible domains, the DEC Hodge--Laplace approximation superconverges on certain symmetric meshes~\cite{guzman2025}. Whether that gain survives on the closed manifolds and non-well-centered tessellations used here, whose harmonic forms are non-trivial, remains open.

The scope of the block-diagonalization extends well beyond the linear, isotropic model problems of the experiments. As clarified in Remarks~\ref{rmk:anisotropic_data} and~\ref{rmk:nonlinear}, it is a property of the operator and therefore persists when the boundary data and forcing are fully anisotropic, and it extends to the linear substeps of operator-splitting schemes for nonlinear equations such as Navier--Stokes.

Several limitations remain. The framework requires the mesh to exactly realize the symmetry group $G$; for meshes with only approximate symmetry, the orbit-averaging technique of Remark~\ref{rmk:averaging} restores equivariance at the cost of a one-time geometric preprocessing step, while a more systematic treatment via approximate representation theory or perturbation bounds on the block structure remains an open direction. In three or more dimensions, positivity of the circumcentric Hodge star is a substantive geometric constraint not implied by Delaunay alone~\cite{vanderzee2010,hirani2013}; Section~\ref{sec:exp_3complex} sidesteps this by adopting a centroid-based Hodge star, but extending the framework to general non-well-centered tetrahedralizations requires further analysis of the dual-cell volume formula. Finally, assembling the isotypic projections costs $\mathcal{O}(|G|\cdot N_k^2)$; although dominated by the $\mathcal{O}(N_k^3)$ solve for large $N_k$, this may become a bottleneck for very large meshes without sparse-matrix acceleration.

\section{Conclusion}
\label{sec:conclusion}

We have proved that the discrete exterior derivative $d$ and the discrete Hodge star $\star$ are equivariant under isometric finite group actions on simplicial complexes of arbitrary topological dimension. As a direct consequence, any operator assembled from $d$ and $\star$ automatically inherits a block-diagonal structure when expressed in a symmetry-adapted basis, without requiring a separate equivariance analysis for each PDE. A single symmetry analysis of the mesh simultaneously decouples all DEC-based operators defined on it.

Three properties of the framework distinguish it from existing symmetry-exploitation strategies. First, it applies natively to curved manifolds where global spectral bases are unavailable, such as geodesic spheres and tori, as well as to three-dimensional tessellations. Second, the symmetry analysis is purely combinatorial. The symmetry-adapted basis is computed once from the group action on the mesh, independently of the specific PDE coefficients or boundary data. Third, the block structure provides insight beyond the speedup. Eigenvalues of DEC operators are automatically classified by irreducible representation, giving each eigenfunction a definite symmetry label (Experiment~2) and connecting the discrete computation to the spectral theory of the symmetry group.

A practical consequence of the dimension-agnostic proof of Theorem~\ref{thm:d_equiv} is that the framework extends transparently to three-dimensional simplicial complexes, opening symmetry-based DEC parallelization to crystallographic and molecular-point-group settings. The Hodge-star equivariance (Theorem~\ref{thm:star_equiv}) holds under the same isometry hypothesis. The additional structural condition in three dimensions concerns the invertibility of $\star$ rather than its equivariance. Positivity of the circumcentric star is automatic on Delaunay surfaces but not on tetrahedral meshes~\cite{vanderzee2010,hirani2013}. The fan-BCC mesh of Section~\ref{sec:exp_3complex} demonstrates that this requirement can be circumvented in practice by replacing the circumcentric Hodge with a centroid-based variant, for which the equivariance argument is unchanged.

The algebraic reduction developed here is complementary to existing domain-decomposition implementations of DEC~\cite{boom2022}, and the composition of the two levels of parallelism, block-diagonalization by symmetry followed by MPI distribution within each sufficiently large block, is a natural engineering follow-up. Exploiting the full lattice symmetry at the block-diagonalization layer requires a symmetry-preserving mesh subdivision, which the default 5-/6-tetrahedron brick pattern of~\cite{boom2022} is not, by the same diagonal-versus-symmetry argument applied to the fan-BCC rule in Section~\ref{sec:exp_3complex}. The symmetry level itself admits a further algebraic reduction. Because Schur's lemma makes each isotypic block a tensor product $M_\mu \otimes I_{d_\mu}$, only one block per irreducible representation carries distinct spectral information, which deepens the sequential speedup and shrinks the largest block passed to the distributed layer (Section~\ref{sec:exp_3complex}). Extending this unified equivariance framework to non-compact groups, adaptive meshes, and higher-order DEC discretizations~\cite{schulz2018} remains an open direction for future research.

\section*{Code and data availability}
The code, recorded results, and figures that reproduce the numerical experiments reported here are openly available under the MIT license at \url{https://github.com/ldsufrpe/dec-equivariance}. The repository documents which script produces each figure and table in the paper.

\bibliographystyle{unsrtnat}
\bibliography{references}

\end{document}

%% file: figures/icosahedron_mesh.tex

\coordinate (IcoV1)  at ( 0,      1,      1.61803);
\coordinate (IcoV2)  at ( 0,      1,     -1.61803);
\coordinate (IcoV3)  at ( 0,     -1,      1.61803);
\coordinate (IcoV4)  at ( 0,     -1,     -1.61803);
\coordinate (IcoV5)  at ( 1,      1.61803, 0);
\coordinate (IcoV6)  at ( 1,     -1.61803, 0);
\coordinate (IcoV7)  at (-1,      1.61803, 0);
\coordinate (IcoV8)  at (-1,     -1.61803, 0);
\coordinate (IcoV9)  at ( 1.61803, 0,      1);
\coordinate (IcoV10) at ( 1.61803, 0,     -1);
\coordinate (IcoV11) at (-1.61803, 0,      1);
\coordinate (IcoV12) at (-1.61803, 0,     -1);

\tikzset{hiddenface/.style={fill=gray!8, draw=gray!35, thin,
         dash pattern=on 2pt off 1.5pt}}
\fill[hiddenface] (IcoV4)--(IcoV10)--(IcoV2)--cycle;   
\fill[hiddenface] (IcoV4)--(IcoV2)--(IcoV12)--cycle;   
\fill[hiddenface] (IcoV4)--(IcoV12)--(IcoV8)--cycle;   
\fill[hiddenface] (IcoV4)--(IcoV8)--(IcoV6)--cycle;    
\fill[hiddenface] (IcoV4)--(IcoV6)--(IcoV10)--cycle;   
\fill[hiddenface] (IcoV11)--(IcoV12)--(IcoV8)--cycle;  
\fill[hiddenface] (IcoV11)--(IcoV8)--(IcoV3)--cycle;   
\fill[hiddenface] (IcoV3)--(IcoV8)--(IcoV6)--cycle;    
\fill[hiddenface] (IcoV3)--(IcoV6)--(IcoV9)--cycle;    
\fill[hiddenface] (IcoV9)--(IcoV6)--(IcoV10)--cycle;   

\tikzset{visface/.style={fill=gray!28, draw=gray!65, line width=0.55pt}}
\fill[visface] (IcoV5)--(IcoV9)--(IcoV10)--cycle;    
\fill[visface] (IcoV5)--(IcoV10)--(IcoV2)--cycle;    
\fill[visface] (IcoV1)--(IcoV3)--(IcoV9)--cycle;     
\fill[visface] (IcoV1)--(IcoV11)--(IcoV3)--cycle;    
\fill[visface] (IcoV7)--(IcoV2)--(IcoV12)--cycle;    
\fill[visface] (IcoV1)--(IcoV9)--(IcoV5)--cycle;     
\fill[visface] (IcoV7)--(IcoV12)--(IcoV11)--cycle;   
\fill[visface] (IcoV5)--(IcoV2)--(IcoV7)--cycle;     
\fill[visface, fill=gray!38] (IcoV1)--(IcoV7)--(IcoV11)--cycle; 
\fill[visface, fill=gray!45] (IcoV1)--(IcoV5)--(IcoV7)--cycle;  

\foreach \v in {IcoV1,IcoV2,IcoV3,IcoV5,IcoV7,IcoV9,IcoV10,IcoV11,IcoV12}{
    \fill[gray!80] (\v) circle (1.8pt);
}